\documentclass[a4paper, 11pt]{article}

\usepackage[utf8]{inputenc}
\usepackage[T1]{fontenc}
\usepackage[a4paper, margin=3cm]{geometry}
\usepackage{needspace}

\usepackage[english]{babel}

\usepackage{libertine} 
\usepackage{inconsolata} 

\usepackage{amsmath, amsthm, amssymb, amsfonts}
\usepackage{stmaryrd}
\usepackage{graphicx}
\usepackage{enumitem}
\usepackage{authblk}
\usepackage[textsize=small, textwidth=2cm, color=orange]{todonotes}
\usepackage{thmtools}
\usepackage{thm-restate}
\usepackage[colorlinks=true, citecolor=green]{hyperref}
\usepackage[noabbrev,capitalize]{cleveref}
\usepackage{float}
\usepackage{xspace}
\usepackage{tikz}
\usetikzlibrary{scopes}
\tikzset{myvert/.style={draw=black,circle,inner sep=3pt}}

\usepackage{caption,subcaption}

\usepackage[ruled,vlined]{algorithm2e}

\SetCommentSty{mycommfont} 

\renewenvironment{abstract}
{\small\vspace{-1em}
\begin{center}
\bfseries\abstractname\vspace{-.5em}\vspace{0pt}
\end{center}
\list{}{
\setlength{\leftmargin}{0.6in}%
\setlength{\rightmargin}{\leftmargin}}%
\item\relax} {\endlist}

\renewcommand\qedsymbol{$\blacksquare$}

\def\ccqedsymbol{\ifmmode$\lrcorner$\else{\unskip\nobreak\hfil
\penalty50\hskip1em\null\nobreak\hfil$\lrcorner$
\parfillskip=0pt\finalhyphendemerits=0\endgraf}\fi}

\usepackage[linewidth=1pt,
  middlelinecolor= black, middlelinewidth=0.4pt, roundcorner=1pt, topline = false,
  rightline = false, bottomline = false, rightmargin=0pt, skipabove=5pt,
  skipbelow=5pt, leftmargin=-.3cm, innerleftmargin=.3cm, innerrightmargin=0pt,
  innertopmargin=0pt, innerbottommargin=0pt ] {mdframed}
\renewenvironment{proof}[1][] {\begin{mdframed} \emph{Proof\xspace#1.~}}
                 {\qedsymbol\end{mdframed}}

\declaretheorem[name=Theorem, numberwithin=section]{theorem}
\declaretheorem[name=Lemma, sibling=theorem]{lemma}
\declaretheorem[name=Proposition, sibling=theorem]{proposition}

\declaretheorem[name=Corollary, sibling=theorem]{corollary}

\crefname{conjecture}{Conjecture}{Conjectures}

\crefname{hypothesis}{Hypothesis}{Hypothesis}

\declaretheorem[name=Claim]{claim}
\crefname{claim}{Claim}{Claims}

\declaretheorem[name=Observation, sibling=theorem]{observation}
\def\cqedsymbol{\ifmmode$\lrcorner$\else{\unskip\nobreak\hfil
    \penalty50\hskip1em\null\nobreak\hfil$\lrcorner$
\parfillskip=0pt\finalhyphendemerits=0\endgraf}\fi}

\crefname{algocf}{Algorithm}{Algorithms}

\let\le\leqslant \let\ge\geqslant \let\leq\leqslant \let\geq\geqslant

\renewcommand\qedsymbol{$\blacksquare$}
\newcommand{\Kcg}{Kempe change\xspace}
\newcommand{\Kcgs}{Kempe changes\xspace}
\newcommand{\Kcn}{Kempe chain\xspace}

\DeclareMathOperator{\diam}{diam}
\DeclareMathOperator{\tw}{tw}
\DeclareMathOperator{\mad}{mad}

\title{Kempe changes in degenerate graphs}

\author[1]{Marthe Bonamy}
\author[1]{Vincent Delecroix}
\author[1]{Clément Legrand--Duchesne}
\affil[1]{LaBRI, CNRS, Université de Bordeaux, Bordeaux, France.}

\begin{document}

\maketitle

\begin{center}
  \begin{abstract}
We consider Kempe changes on the $k$-colorings of a graph on $n$ vertices. If
the graph is $(k-1)$-degenerate, then all its $k$-colorings are equivalent up to
Kempe changes. However, the sequence between two $k$-colorings that arises from
the proof may be exponential in the number of vertices. An intriguing open
question is whether it can be turned polynomial. We prove this to be possible
under the stronger assumption that the graph has treewidth at most
$k-1$. Namely, any two $k$-colorings are equivalent up to $O(k n^2)$ Kempe
changes. We investigate other restrictions (list coloring, bounded maximum
average degree, degree bounds). As a main result, we derive that given an
$n$-vertex graph with maximum degree $\Delta$, the $\Delta$-colorings are all
equivalent up to $O(n^2)$ Kempe changes, unless $\Delta=3$ and some connected
component is a $3$-prism.
  
  \end{abstract}
\end{center}

Given a $k$-colored graph, a \emph{\Kcn} is a connected component in the
subgraph induced by two given colors. A \Kcg consists in swapping the two colors
in a Kempe chain, thereby obtaining a new $k$-coloring of the graph. Two
$k$-colorings of a graph are \emph{Kempe equivalent} if one can be obtained from
the other through a series of Kempe changes. This elementary operation on the
$k$-colorings of a graph was introduced by Kempe in 1879, in an unsuccessful
attempt to prove the four color theorem~\cite{kempe1879geographical}.

The study of Kempe changes has a vast history, see e.g.~\cite{mohar2006kempe}
for a comprehensive overview or~\cite{bonamy2019conjecture} for a recent result
on general graphs. We refer the curious reader to the relevant chapter of a 2013
survey by Van Den Heuvel~\cite{van2013complexity}. Kempe equivalence falls
within the wider setting of combinatorial reconfiguration,
which~\cite{van2013complexity} is also an excellent introduction to. Beyond the
intrinsic interest of theoretical results, the study of Kempe changes is
motivated by practical applications in statistical physic and approximate
counting of colorings (see e.g.\ \cite{sokal2000personal,mohar2009new} for nice
overviews). Closer to graph theory, Kempe equivalence can be studied with a goal
of obtaining a random coloring by applying random walks and rapidly mixing
Markov chains, see e.g.~\cite{vigoda}.

Kempe changes were introduced as a mere tool, and are decisive in the proof of
Vizing's edge coloring theorem~\cite{vizing1964estimate}. However, the
equivalence class they define on the set of $k$-colorings is itself highly
interesting. In which cases is there a single equivalence class? In which cases
does every equivalence class contain a coloring that uses the minimum number of
colors? Vizing conjectured in 1965~\cite{vizing1968some} that the second
scenario should be true in every line graph, no matter the choice of $k$.

Our main interest is the study of the reconfiguration graph $R^k(G)$, whose
vertices are the $k$-colorings of $G$ and in which two colorings are adjacent if
and only if they differ by one Kempe change. We are interested in the following
questions: in which setting is the reconfiguration graph connected, that is, any
two $k$-colorings are Kempe equivalent ? When this is the case, can we bound the
length of the shortest sequence of Kempe changes between any two colorings,
i.e. the diameter of the reconfiguration graph ?

A graph is said to be $d$-degenerate if all its (non-empty) subgraphs contain a
vertex of degree at most $d$. Note that if $G$ has maximum degree $\Delta$ then
it is trivially $\Delta$-degenerate, and even $(\Delta-1)$-degenerate if $G$ is
not regular.  Las Vergnas and Meyniel~\cite{lasvergnas1981kempe} showed in 1981
that there exists a sequence of Kempe changes between any two $k$-colorings of a
$d$-degenerate graph $G$ when $k \ge d+1$. In other words, the corresponding
reconfiguration graph is connected. However, the sequence they provided may have
exponential length.

Reconfiguration restricted to trivial Kempe changes --- Kempe changes involving
only one vertex --- is another well-studied topic, known as vertex recoloring. The
lemma of Las Vergnas and Meyniel echoes a result proved by
Cereceda~\cite{cercedas2007mixing} in the setting of vertex recoloring: all the
$k$-colorings of a $d$-degenerate graph are equivalent up to trivial Kempe
changes when $k \ge d+2$. However, the sequence between two $k$-colorings that
arises from the corresponding proof may once again be exponential in the number
$n$ of vertices. Cereceda conjectured that there exists one of length
$O(n^2)$. In a breakthrough paper, Bousquet and Heinrich proved that there
exists a sequence of length $O(n^{d+3})$~\cite{bousquet2019polynomial}. However,
Cereceda's conjecture remains open, even for $d = 2$.

Obtaining similar bounds with regular Kempe changes on $d$-degenerate graphs
with one fewer color would have many consequences, as the lemma of Las Vergnas
and Meyniel is used as a base ground in several proofs. Unfortunately, the bound
of Bousquet and Heinrich~\cite{bousquet2019polynomial} does not extend to this
setting, and even a polynomial upper-bound on the number of changes would be
highly interesting. With this in mind, we prove three polynomial bounds on the
diameter of the reconfiguration graph in closely related settings.

\subsection*{Kempe equivalence of $\Delta$ colorings}
Any graph $G$ can be greedily colored with $(\Delta+1)$ colors, where $\Delta$
is the maximum degree of $G$.  Brooks' theorem states that if $G$ is not a
clique or an odd cycle, then $\Delta$ colors suffice. Many different proofs of
this theorem exist, see~\cite{carston2014brooks} for a collection of proofs of
Brooks' theorem using various techniques --- for that matter, some of them using
Kempe changes.

In the more restrictive setting of Brooks' theorem, Mohar~\cite{mohar2006kempe}
conjectured that all the $k$-colorings of a graph are Kempe equivalent for $k
\ge \Delta$. Note that the result of Las Vergnas and
Meyniel~\cite{lasvergnas1981kempe} settles the case of non-regular
graphs. Feghali \emph{et al.}~\cite{feghali2017kempe} proved that the conjecture
holds for all cubic graphs but the 3-prism (see Fig.~\ref{fig:3prism}). The
conjecture does not hold for the 3-prism, as can be seen through the lenses of
\emph{frozen} colourings. A colouring is frozen if any bichromatic subgraph is
connected, hence performing any Kempe change leaves the color partition
unchanged. Since the 3-prism admits two frozen $3$-colorings with different
color partitions, they are not Kempe equivalent. To our knowledge, this is the
only argument at our disposal to prove that a reconfiguration graph is
disconnected.

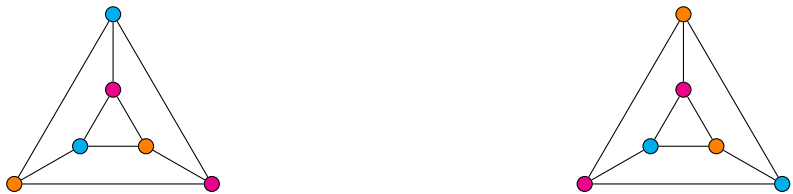
\begin{figure}[h]
    \begin{subfigure}[b]{.5\linewidth}
        \centering
        \begin{tikzpicture}
          \node[draw=black,fill=magenta,circle,inner sep=2pt] (a) at (90:.5cm) {};
          \node[draw=black,fill=cyan,circle,inner sep=2pt] (b) at (210:.5cm) {};
          \node[draw=black,fill=orange,circle,inner sep=2pt] (c) at (330:.5cm)
               {};

          \node[draw=black,fill=cyan,circle,inner sep=2pt] (A) at (90:1.5cm) {};
          \node[draw=black,fill=orange,circle,inner sep=2pt] (B) at (210:1.5cm) {};
          \node[draw=black,fill=magenta,circle,inner sep=2pt] (C) at
               (330:1.5cm) {};

          \draw (a) -- (A); \draw (b) -- (B); \draw (c) -- (C); \draw (a) -- (b)
          -- (c) -- (a); \draw (A) -- (B) -- (C) -- (A);
        \end{tikzpicture}
    \end{subfigure}%
    \begin{subfigure}[b]{.5\linewidth}
        \centering
        \begin{tikzpicture}
          \node[draw=black,fill=magenta,circle,inner sep=2pt] (a) at (90:.5cm) {};
          \node[draw=black,fill=cyan,circle,inner sep=2pt] (b) at (210:.5cm) {};
          \node[draw=black,fill=orange,circle,inner sep=2pt] (c) at (330:.5cm)
               {};

          \node[draw=black,fill=orange,circle,inner sep=2pt] (A) at (90:1.5cm)
               {};
          \node[draw=black,fill=magenta,circle,inner sep=2pt] (B) at
          (210:1.5cm) {};
          \node[draw=black,fill=cyan,circle,inner sep=2pt]
               (C) at (330:1.5cm) {};

          \draw (a) -- (A); \draw (b) -- (B); \draw (c) -- (C); \draw (a) -- (b)
          -- (c) -- (a); \draw (A) -- (B) -- (C) -- (A);
        \end{tikzpicture}
    \end{subfigure}
    \caption{Two frozen 3-colorings of the 3-prism}
    \label{fig:3prism}
\end{figure}

Bonamy \emph{et al.}~\cite{bonamy2019conjecture} later showed that the
conjecture also holds for $\Delta$-regular graphs with $\Delta \ge 4$. Both
paper heavily rely on the lemma of Las Vergnas and Meyniel and provide sequences
that possibly have exponential length.

In~\cref{sec:delta}, we give a polynomial upper bound on the diameter of the
reconfiguration graph in this setting.
\begin{restatable}{theorem}{thmdelta}\label{thm:delta}
  All the $k$-colorings of an $n$-vertex graph $G$ with maximum degree at most
  $\Delta \le k$ are equivalent up to $O(n^2)$ Kempe changes, unless $k=3$ and $G$
  is the $3$-prism.
\end{restatable}
The main idea of the proof is to improve the result of Las Vergnas and
Meyniel~\cite{lasvergnas1981kempe} by showing that there exists a sequence of
Kempe changes of length $O(n^2)$ between any two $k$-colorings of a
$d$-degenerate graph when $k \ge d+1$, under the additional assumption that all
the vertices but one have degree at most $d+1$ (see~\cref{sec:blackbox}).

\subsection*{Kempe equivalence in graphs of bounded $\mad$}
The maximum average degree of a graph $G$ is a measure of the sparsity of $G$,
defined as the maximum of the $$\mad(G) = \max_{\emptyset \neq H \subseteq G}
\frac{2|E(H)|}{|V(G)|}.$$ For all $d \ge 1$, if a graph has $\mad$ strictly less
than $d$, then it is $(d-1)$-degenerate: all its subgraph have average degree
less than $d$, so admit a vertex of degree at most $d-1$.

We prove that the $\mad$ of a graph, while related to its degeneracy, proves to
be easier to work with in this setting.

\begin{restatable}{theorem}{thmmad}\label{thm:mad}
  Let $G$ be a graph with $\mad(G) \le k-\varepsilon$. All the $k$-colorings of $G$
  are Kempe equivalent up to $O(Poly_\varepsilon(n))$ Kempe changes.
\end{restatable}

We prove the above theorem by giving an upper bound on the number of Kempe
changes when lists are involved (see~\cref{sec:blackbox}), and by adapting ideas
developed by Bousquet and Perarnau in~\cite{bousquet2015sparse} in the setting
of single vertex recoloring (see~\cref{sec:mad}).

\subsection*{Kempe equivalence in bounded treewidth graphs}

Another way to strengthen the degeneracy assumption involves the treewidth of a
graph (the treewidth is a graph parameter that measures how close a graph is
from being a tree, see~\cref{sec:def} for a definition). A graph of treewidth
$k$ is $k$-degenerate, while there are $2$-degenerate graphs with arbitrarily
large treewidth. Bonamy and Bousquet~\cite{bonamy2013recoloring} confirmed
Cereceda's conjecture for graphs of treewidth $k$. In~\cref{sec:tw} we extend
this result to non-trivial Kempe changes with one fewer color.

\begin{theorem}\label{thm:tw}
  Let $\tw$ and $n$ be to integer. Any two $k$-colorings of an $n$-vertex graph
  $G$ with treewidth $\tw$ are equivalent up to $O(\tw n^2)$ Kempe changes, when
  $k \ge \tw+1$.
\end{theorem}

Additionally, the proof of Theorem~\ref{thm:tw} is constructive and yields an
algorithm to compute such a sequence in time $f(k)\cdot \textrm{Poly}(n)$. Given
a witness that the graph has treewidth at most $k$, the complexity drops to
$k\cdot \textrm{Poly}(n)$.

\section{Preliminaries}\label{sec:def}
A $k$-\emph{coloring} of a graph $G$ is a map $\alpha \colon V(G) \to [k]$ such
that for every edge $uv \in E(G)$, $\alpha(u) \neq \alpha (v)$. Given a
$k$-coloring $\alpha$ of a graph $G$ and $c \in [k]$, we denote
$K_{u,c}(\alpha,G)$ the connected component of the subgraph of $G$ induced by
the colors $c$ and $\alpha(u)$, that contains $u$. With a slight abuse of
notation, we will also use $K_{u,c}(\alpha,G)$ to denote the corresponding Kempe
change. We may drop the parameter $G$ when there is no ambiguity. When
describing algorithms, we will denote $\widetilde{\alpha}$ (or
$\widetilde{\beta}$, etc) the current coloring, obtained from the original
coloring $\alpha$ (or $\beta$, etc). We denote $\mathcal{C}^k(G)$ the set of
$k$-colorings of $G$ and $R^k(G)$ the reconfiguration graph, whose vertex set is
$\mathcal{C}^k(G)$ and in which two colorings are adjacent if they differ by one
Kempe change.

We denote $N(u)$ and $N[u]$ the open and closed neighborhoods of a vertex $u$,
respectively. Given an ordering $v_1 \prec \ldots \prec v_n$ of the vertices of
$G$, we denote $N^+(v_i) = \{ v_j \in N(v_i) | j > i\}$ and $N^-(v_i) = \{ v_j
\in N(v_i) | j < i\}$.

A graph $G$ is $d$-\emph{degenerate} if all its (non-empty) subgraphs contain a
vertex of degree at most $d$. This is equivalent to admitting a
\emph{$d$-degeneracy sequence}: an ordering of the vertices such that for all $v
\in V$, $N^+(v)$ is of size at most $d$. A $d$-degenerate graph is trivially
$(d+1)$-colorable, by doing a decreasing induction on the vertices.  We will
extensively use the lemma of Las Vergnas and Meyniel:
\begin{lemma}[\cite{lasvergnas1981kempe}]\label{lem:vergnas}
  All $k$-colorings of a $d$-degenerate graph are Kempe equivalent when $k \ge d
  +1$.
\end{lemma}

For completeness, we include a proof of it.
\begin{proof}
  We proceed by induction on the number of vertices. Let $G$ be a $d$-degenerate
  graph on $n$ vertices and $k \ge d+1$. Since $G$ is $d$-degenerate, there
  exists a vertex $v$ of degree at most $d$. Let $G' = G \setminus\{v\}$. Let
  $\alpha$ and $\beta$ be two $k$-colorings of $G$.  We claim that we can apply
  a series of Kempe changes on $\alpha$ in $G$ so as to obtain a coloring
  $\gamma$ whose restriction to $G'$ is equal to $\beta_{|G'}$. Assuming this
  claim, if $\gamma(v) \neq \beta(v)$, then neither $\gamma(v)$ nor $\beta(v)$
  appears in $\gamma(N(v))$ since $\gamma_{|G'} = \beta_{G'}$. As a result,
  performing the trivial Kempe change $K_{v, \beta(v)}(\gamma, G)$ yields
  $\beta$.

  It remains only to prove the above claim. By the induction hypothesis, all
  $k$-colorings of $G'$ are Kempe equivalent, so there exists a sequence $S$ of
  Kempe changes in $G'$ leading from $\alpha_{|G'}$ to $\beta_{|G'}$. Our goal
  is to extend $S$ to a series of Kempe changes in $G$. Each Kempe change
  $K_{u,c}(\widetilde{\alpha}, G')$ in $S$ is either applied directly or
  preceeded by an appropriate trivial Kempe change, as follows.  If
  $\widetilde{\alpha}(v) \not\in \{c,\widetilde{\alpha}(u)\}$ or no neighbor of
  $v$ belongs to the corresponding chain, we directly apply the same change in
  $G$.  Otherwise, the issue is that adding $v$ to $G'$ might result in merging
  two connected bichromatic components of $G'$.  Assume from now on without loss
  of generality that $\widetilde{\alpha}(v)=\widetilde{\alpha}(u)$.

  If $K_{u,c}(\widetilde{\alpha}, G)$ and $K_{u,c}(\widetilde{\alpha}, G')$
  agree except on $v$ (in particular, if $v$ has at most one neighbor of color
  $c$), we also apply directly the same Kempe change.  Indeed, the change will
  impact the color of $v$ in $G$ but the restriction to $G'$ behaves as
  expected.  Therefore, we can assume that $v$ has at least two neighbors
  colored $c$.  Since $v$ has degree at most $d$ and $k \ge d+1$, there exists a
  color $c_v$ that is not used in $N[v]$. After the trivial Kempe change $K_{v,
    c_v}(\widetilde\alpha, G)$, the vertex $v$ is colored $c_v$. We can now
  apply the desired Kempe change.

  This concludes the proof of the claim, hence of the lemma.
\end{proof}

A graph $H$ is \emph{chordal} if every induced cycle is a
triangle. Equivalently, there is an ordering of the vertices such that $N^+[v]$
is a clique for any vertex $v$. As a consequence, the chromatic number $\chi(H)$
of a chordal graph $H$ is equal to the size $\omega(H)$ of a largest clique in
$H$.  Note that we also have $\omega(H) = d +1$ where $d$ is the degeneracy of
$H$.

\begin{proposition}[\cite{bonamy2020diameter}]\label{prop:chordal}
  Given an $n$-vertex chordal graph $G$ and an integer $p$, any two
  $p$-colorings of $G$ are equivalent up to at most $n$ Kempe changes.
\end{proposition}
The proof presented in~\cite{bonamy2020diameter} for~\cref{prop:chordal} is
constructive and the corresponding algorithm runs in linear time.

The \emph{treewidth} of a graph measures how much a graph ``looks like'' a
tree. Out of the many equivalent definitions of treewidth, we use the following:
a graph $G$ has treewidth at most $k$ if there exists a chordal graph $H$ such
that $G$ is a (not necessarily induced) subgraph of $H$, with $\omega(H) =
\chi(H) \le k+1$. For example, if $G$ is a tree, we note that $G$ is chordal and
$2$-colorable. By taking $H=G$, we derive that $G$ has treewidth at most $1$.
If $G$ has treewidth at most $k$, there is a chordal graph $H$ with
$\omega(H)=k+1$ that admits $G$ as a subgraph. Therefore, $H$ and thus $G$ are
$k$-degenerate.  The converse does not hold, as rectangular grids have
degeneracy 2 yet unbounded treewidth.

\section{Kempe recoloring with list assignments}\label{sec:blackbox}
Let $G$ be a graph and $L: V(G) \to \mathcal{P}(\mathbb{N})$ be a list
assignment.  A coloring $\alpha$ of $G$ is an \emph{$L$-coloring} if $\alpha(u)
\in L(u)$ for every vertex $u$.  Kempe changes can be defined as before,
although the resulting coloring may not be an $L$-coloring. Possible
obstructions will be defined later on.

We strengthen the lemma of Las Vergnas and Meyniel under two different
additional assumption:

\begin{proposition}\label{prop:blackbox}
    Let $G$ be a graph and let $v_1 \prec \ldots \prec v_n$ be an ordering of
    $V(G)$.
  \begin{enumerate}
      \item If the ordering yields a $(d-1)$-degeneracy sequence and
        $\deg(v_i)\leq d$ for every $i< n$, then any two $k$-colorings of $G$
        are Kempe equivalent up to $O(n^2)$ Kempe changes, for $k \ge d$. More
        precisely, the color of each vertex is changed a linear number of times.
  \item Given a list assignment $L$ of $G$, if $|L(v_i)| \ge \deg(v_i)+1$ for
    all $i < n$ then any two $L$-colorings of $G$ are equivalent up to $O(n)$
    Kempe changes.
  \end{enumerate}
\end{proposition}

Let $\alpha$ be a $L$-coloring of $G$, $v$ a vertex and a color $c \in L(v)
\setminus \alpha(N^+(v))$. Let $K = K_{v,c}(\alpha,G)$ We say that a vertex $u
\in K \setminus \{v\}$ is \emph{blocking} the Kempe change if $c$ or $\alpha(v)$
does not belong to $L(u)$. If a Kempe chain admits at least one blocking vertex,
then we say it is \emph{not feasible}.  To further analyze how possible
obstructions can appear, we introduce the following definitions.  We say that a
vertex $u \in K \setminus \{v\}$ is
\begin{itemize}
\item \emph{branching} if $u$ has degree at least 3 in the Kempe chain $K$,
\item \emph{problematic} if $u$ has at least two greater neighbors in the Kempe
  chain $K$.
\end{itemize}
We say that $u$ is a \emph{bad vertex} for $K$ if it is either branching,
problematic or blocking $K$. It is a \emph{first bad vertex} if there exists a
path from $v$ to $u$ in $K$ that contains no other bad vertex than $u$. Note
that a first bad vertex is necessarily smaller than $v$.

We use the same algorithm to prove both parts of~\cref{prop:blackbox}.
\begin{algorithm}[!ht]
  \SetAlgoLined \SetKwInOut{Input}{Input} \SetKwInOut{Output}{Output}
  \caption{}
  \label{alg:list}
  \Input{An $L$-coloring $\alpha$ of $G$, a vertex $v$, a color $c \in L(v)
    \setminus \alpha(N^+[v])$.}
  \Output{An $L$-coloring $\beta$ of $G$ which agrees with $\alpha$ on
    $\{w | w \succ v \}$, with $\beta(v) = c$.}
  Let $\beta = \alpha$\;
  \nl\label{alg:list:while}\While{$\exists u \in K_{v,c}(\beta, G) \setminus \{v\}$ that
    is bad}
        { \nl\label{alg:bad}Let $u$ be the greatest first bad vertex\;
          \nl\label{alg:list:if}\uIf{ there exists a color $c'$ in $L(u) \setminus \beta(N[u])$}
              { \tcp{In particular, this is the case if $u$ is blocking or
                  branching.}
                Let $\beta$ be the result of the trivial Kempe change $K_{u,c'}(\beta, G)$\;}
          \nl\label{alg:list:else}\uElse{
                Let $c'$ be a color in $L(u) \setminus \beta(N^+[u])$\;
                \tcp{The set is non-empty as $u$ is problematic.}
                Let $\beta$ be the result of~\cref{alg:list} on input $(\beta,
                 u, c')$\;}
        }
        Perform $K_{v,c}(\beta, G)$ in $\beta$\;
  Return $\beta$\;
\end{algorithm}

\subsection{Recoloring $(d-1)$-degenerate graphs with an additional assumption
  on the degree}

Assume that $G$ is $(d-1)$-degenerate, with corresponding degeneracy sequence
$v_1 \prec \dots \prec v_n$, and that all the vertices but possibly $v_n$ have
degree at most $d$. To prove the first point of ~\cref{prop:blackbox}, we prove
the following lemma:
\begin{lemma}\label{lem:claim}
    Let $\alpha$ be a $k$-coloring of $G$, let $1 \leq j \leq n$ and $1 \leq c
    \leq k$ such that $c \not\in \alpha(N^+(v_j))$.  \cref{alg:list} applied on
    $(\alpha,v_j,c)$ with list assignment $L(v_i)=\{1,\dots,k\}$ for every $i$
    returns a $k$-coloring $\beta$ such that $\beta(v_j) = c$ and $\forall \ell
    > j, \beta(v_\ell) = \alpha(v_\ell)$. Furthermore, the smaller vertices are
    recolored at most $|N^-(v_j)|$ times, $v_j$ is recolored at most once, and
    the bigger vertices are not recolored.
\end{lemma}
\begin{proof}
  We proceed by induction on $j$. If $j=1$, then the algorithm does not enter
  the while loop, and recolors $v_j$ with a trivial Kempe change. Assume now
  that $1 < j < n$.~\cref{fig:example} shows an example of the execution
  of~\cref{alg:list}. Since $L$ is constant, there is no blocking vertex.

    Consider a $k$-coloring $\gamma$ of $G$ that agrees with $\alpha$ on all
    vertices greater or equal than $v_j$.  Assume that $K_{v_j,c}(\gamma, G)
    \setminus \{v_j\}$ contains a bad vertex, in other words the condition of
    the while loop is satisfied for $\gamma$. Let $u$ be the greatest bad vertex
    as in the line~\ref{alg:bad}.  For every neighbor $w \in N^-(v_j)$ such that
    $\gamma(w) = c$ we define
    $$C_\gamma(w) = K_{w, \gamma(v_j)}(\gamma, G \setminus\{v_j\}) \cup
    \{v_j\}.$$ We have $K_{v_j,c}(\gamma,G) = \bigcup_{w \in N^-(v_j) |
      \gamma(w) = c} C_\gamma(w)$.  A neighbor $w \in N^-(v_j)$ is a \emph{safe
    neighbor in $\gamma$} if $w \succ u$ and either $\gamma(w) \not= c$ or
    $\gamma(w) = c$ and $C_\gamma(w)$ is a decreasing path ending at a vertex
    greater than $u$. A neighbor $w \in N^-(v_j)$ is \emph{unsafe} if it is not
    safe.

\begin{claim}\label{cl:gammadecreases}
  An iteration of the while loop on $\gamma$ results in a coloring $\gamma'$
  with more safe neighbors. Moreover, either $K_{v_j,c}(\gamma',G)$ has no bad
  vertex or the greatest bad vertex $u'$ is smaller than the greatest bad vertex
  $u$ for $K_{v_j,c}(\gamma,G)$.
\end{claim}

\begin{proof}
    The Kempe changes in the while loops operate on vertices smaller or equal
    than $u$. Hence safe neighbors in $\gamma$ remain safe in $\gamma'$ if $u'
    \prec u$. We only need to show that $u' \prec u$ and that some unsafe
    neighbor becomes safe.
    
    We first prove that $u' \prec u$. In either case of the if condition, only
    the vertices smaller or equal than $u$ are modified. Furthermore $u$ is not
    bad anymore in $\gamma'$. Hence either $\gamma'$ has no bad vertex or the
    greatest first one is smaller than $u$.
    
    We now prove that some unsafe neighbor becomes safe.  Let $w \in N^-(v_j)$
    be on a shortest path $P$ from $v_j$ to $u$ in $K_{v_j,c}(\gamma,G)$. Then
    $w$ is an unsafe neighbor with $\gamma(w) = c$. We prove that $w$ is safe in
    $\gamma'$.

    If $u=w$ then $w$ is safe in $\gamma'$ since it is recolored with a color
    $c'$ different from $\{c,\gamma(v_j)\}$.  If $u \not= w$, then at the end of
    the while loop the vertex $u$ is recolored and $C_{\gamma'}(w) = P \setminus
    \{u\}$.  In particular, since $u' \prec u$, $w$ is safe in $\gamma'$.
    \end{proof}

Apply Claim~\ref{cl:gammadecreases} to $\alpha$, then to the coloring obtained
after each new iteration of the while loop (if any). If the resulting coloring
has only safe neighbors, then $K$ is a generalized star, rooted at $v_j$, in
which each branch is decreasing. In such a scenario, the while condition is not
satisfied. The Kempe change $K$ colors $v_j$ with $c$ and does not affect the
colors of the vertices bigger than $v_j$, as desired.

Since the number of safe neighbors increases at each step, the number of
iterations is at most the number of unsafe neighbors in $\alpha$, which is at
most $|N^-(v_j)|$. Note that bigger vertices are not recolored. It remains to
argue more carefully that smaller vertices are not recolored too many times. We
observe that in each iteration of the while loop, a smaller vertex is recolored
at most once. This is trivial if $u$ satisfies the if condition on
line~\ref{alg:list:if}.  If it does not, we observe that $u$ has at most one bad
neighbor, and apply the induction hypothesis to obtain that every vertex smaller
than $u$ is recolored at most once. The conclusion follows.
\end{proof}

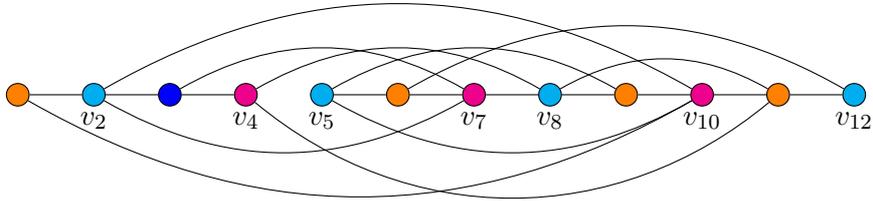
\begin{figure}[!ht]
\begin{center}
  \begin{tikzpicture}
      \node[anchor=north] (graph1) at (0,0) {
      \begin{tikzpicture}
        \node[myvert,fill=cyan] (A) {};
        \node[myvert,fill=orange] (B) [left of=A] {};
        \node[myvert,fill=magenta] (C) [left of=B] {};
        \node[myvert,fill=orange] (D) [left of=C] {};
        \node[myvert,fill=cyan] (E) [left of=D] {};
        \node[myvert,fill=magenta] (F) [left of=E] {};
        \node[myvert,fill=orange] (G) [left of=F] {};
        \node[myvert,fill=cyan] (H) [left of=G] {};
        \node[myvert,fill=magenta] (I) [left of=H] {};
        \node[myvert,fill=blue] (J) [left of=I] {};
        \node[myvert,fill=cyan] (K) [left of=J] {};
        \node[myvert,fill=orange] (L) [left of=K] {};
        
        \node[below=.09] at (A) {$v_{12}$};
        \node[below=.09] at (C) {$v_{10}$};
        \node[below=.09] at (E) {$v_8$};
        \node[below=.09] at (F) {$v_7$};
        \node[below=.09] at (H) {$v_5$};
        \node[below=.09] at (I) {$v_4$};
        \node[below=.09] at (K) {$v_2$};
        
        \draw (A) -- (B) -- (C) -- (D) -- (E) -- (F) -- (G) -- (H);
        \draw (I) -- (J) -- (K) -- (L);

        \path
        (E) edge[bend left] (B)
        (H) edge[bend left] (D)
        (H) edge[bend right] (C)
        (K) edge[bend right] (F)
        (I) edge[bend left] (E)
        (I) edge[bend right=40] (B)
        (G) edge[bend left] (A)
        (J) edge[bend left] (F)
        (K) edge[bend left] (C)
        (L) edge[bend right] (C);
    \end{tikzpicture}};
    
    \node[anchor=north] (legend1) at (0,-3.2) {
    \begin{minipage}{0.8\textwidth}
Initial coloring $\alpha$. The aim is to recolor $v_8$ in pink. The bad
neighbors are $v_4$ and $v_7$. $K(v_4)$ consists only of $v_4$ and $v_8$, and is
a decreasing path. $K(v_7)$ consists of $v_2,v_5,v_7,v_8,v_{10}$. No vertex is
branching, but $v_2$ is problematic.  All colors appear in $N[v_2]$, so we apply
induction on $v_2$ with color orange.
\end{minipage}};
        
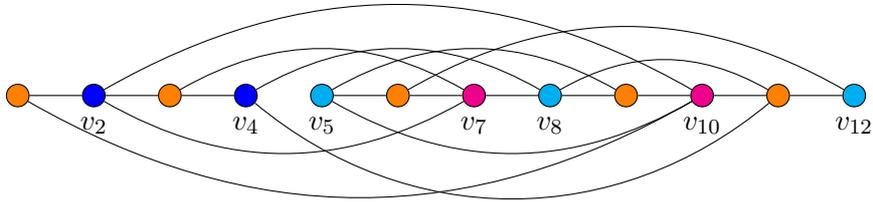
\node[anchor=north] (graph2) at (0,-5.5) {
      \begin{tikzpicture}
        \node[myvert,fill=cyan] (A) {};
        \node[myvert,fill=orange] (B) [left of=A] {};
        \node[myvert,fill=magenta] (C) [left of=B] {};
        \node[myvert,fill=orange] (D) [left of=C] {};
        \node[myvert,fill=cyan] (E) [left of=D] {};
        \node[myvert,fill=magenta] (F) [left of=E] {};
        \node[myvert,fill=orange] (G) [left of=F] {};
        \node[myvert,fill=cyan] (H) [left of=G] {};
        \node[myvert,fill=blue] (I) [left of=H] {};
        \node[myvert,fill=orange] (J) [left of=I] {};
        \node[myvert,fill=blue] (K) [left of=J] {};
        \node[myvert,fill=orange] (L) [left of=K] {};

        \node[below=.15] at (A) {$v_{12}$};
        \node[below=.15] at (C) {$v_{10}$};
        \node[below=.15] at (E) {$v_8$};
        \node[below=.15] at (F) {$v_7$};
        \node[below=.15] at (H) {$v_5$};
        \node[below=.15] at (I) {$v_4$};
        \node[below=.15] at (K) {$v_2$};
        
        \draw (A) -- (B) -- (C) -- (D) -- (E) -- (F) -- (G) -- (H);
        \draw (I) -- (J) -- (K) -- (L);

        \path
        (E) edge[bend left] (B)
        (H) edge[bend left] (D)
        (H) edge[bend right] (C)
        (K) edge[bend right] (F)
        (I) edge[bend left] (E)
        (I) edge[bend right=40] (B)
        (G) edge[bend left] (A)
        (J) edge[bend left] (F)
        (K) edge[bend left] (C)
        (L) edge[bend right] (C);
    \end{tikzpicture}};
    
\node[anchor=north] (legend2) at (0,-8.7) {
\begin{minipage}{0.8\textwidth}
      $K_{v_8, \text{pink}}(\widetilde\beta)$ is now a generalized
        star, centered at $v_8$, whose branches are decreasing paths.\label{fig:example:c}.
\end{minipage}
};

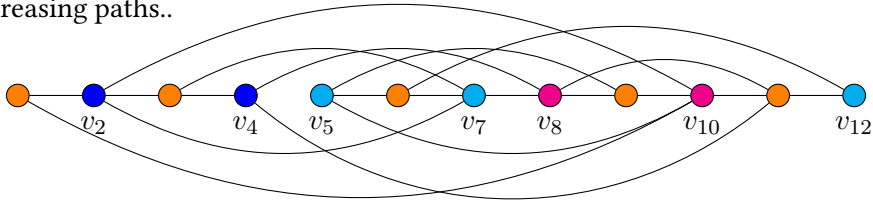
\node[anchor=north] (graph3) at (0,-9) {
      \begin{tikzpicture}
        \node[myvert,fill=cyan] (A) {};
        \node[myvert,fill=orange] (B) [left of=A] {};
        \node[myvert,fill=magenta] (C) [left of=B] {};
        \node[myvert,fill=orange] (D) [left of=C] {};
        \node[myvert,fill=magenta] (E) [left of=D] {};
        \node[myvert,fill=cyan] (F) [left of=E] {};
        \node[myvert,fill=orange] (G) [left of=F] {};
        \node[myvert,fill=cyan] (H) [left of=G] {};
        \node[myvert,fill=blue] (I) [left of=H] {};
        \node[myvert,fill=orange] (J) [left of=I] {};
        \node[myvert,fill=blue] (K) [left of=J] {};
        \node[myvert,fill=orange] (L) [left of=K] {};

        \node[below=.15] at (A) {$v_{12}$};
        \node[below=.15] at (C) {$v_{10}$};
        \node[below=.15] at (E) {$v_8$};
        \node[below=.15] at (F) {$v_7$};
        \node[below=.15] at (H) {$v_5$};
        \node[below=.15] at (I) {$v_4$};
        \node[below=.15] at (K) {$v_2$};

        \draw (A) -- (B) -- (C) -- (D) -- (E) -- (F) -- (G) -- (H);
        \draw (I) -- (J) -- (K) -- (L);

        \path
        (E) edge[bend left] (B)
        (H) edge[bend left] (D)
        (H) edge[bend right] (C)
        (K) edge[bend right] (F)
        (I) edge[bend left] (E)
        (I) edge[bend right=40] (B)
        (G) edge[bend left] (A)
        (J) edge[bend left] (F)
        (K) edge[bend left] (C)
        (L) edge[bend right] (C);
    \end{tikzpicture}
};

\node[anchor=north] (legend3) at (0,-12.2) {
\begin{minipage}{0.8\textwidth}
  The vertex $v_8$ has been recolored pink, after a total of 2 Kempe changes.
\end{minipage}
};
\end{tikzpicture}\end{center}

    \caption{Example of recoloring on a 3-degenerate graph, following the steps of the
      proof\label{fig:example} of \cref{lem:claim}.}
  \end{figure}

\subsection{List recoloring}
We now prove the second point of~\cref{prop:blackbox}.  Let $L$ be a list
assignment of $G$ such that $|L(v_i)| \ge \deg(v_i) +1$ for all $i < n$. We
prove the following claim:
\begin{claim}
  Let $\alpha$ be a $L$-coloring of $G$, let $1 \leq j \leq n$ and $c \in L(v)
  \setminus \alpha(N^+(v))$.  \cref{alg:list} applied on $(\alpha, v_j, c)$
  returns an $L$-coloring $\beta$ such that $\beta(v_j) = c$ and $\forall \ell >
  j$, $\beta(v_\ell) = \alpha(v_\ell)$. Furthermore, \cref{alg:list} performs at
  most $|N^-(v_j)|+1$ Kempe changes and each vertex is recolored at most once.
\end{claim}
\begin{proof}
  First note that the condition line~\ref{alg:list:if} is always
  satisfied. Indeed, because $u$ is different from $v_n$ we have $|L(u)| \geq
  \deg(u) + 1$ and
  \begin{itemize}
  \item if $u$ is problematic, $|\beta(N[u])| = 1 + |\beta(N(u))| \le \deg(u)$,
  \item if $u$ is blocking, $|\beta(N[u])| \le \deg(u) +1$ but at least one of
    $c$ and $\alpha(v_j)$ belongs to $\beta(N(u))$ but not to $L(u)$,
  \item if $u$ is branching, $|\beta(N[u])| = 1 + |\beta(N(u))| \le \deg(u) -1$.
  \end{itemize}

  Now the bound on the number of Kempe changes follow from the exact same
  analysis as in the previous proof.
\end{proof}

\subsection{Proof of~\cref{prop:blackbox}}

\begin{proof}[of~\cref{prop:blackbox}]
  Let $\alpha$ and $\beta$ be two colorings satisfying one of the two settings.
  Let $1 \leq j \leq n$ be the largest such that $\alpha(v_j) \neq
  \beta(v_j)$. We note $v = v_j$ and proceed by induction on $j$. Denote $L'(v)
  = L(v) \setminus \alpha(N^+(v))$, where $L(v)=\{1,\ldots,k\}$ in the first
  setting. Note that $L'(v)$ is not empty. For $c$ in $L'(v)$, denote $p_c =
  |\alpha^{-1}(c) \cap N^-(v)| + |\beta^{-1}(c) \cap N^-(v)|$. If $p_c \ge 2$
  for all $c \in L'(v)$, we have $2 \deg^-(v) \geq 2 |L'(v)|$. However,
  $|L'(v)|\geq |L(v)|-\deg^+(v)>\deg^-(v)$, hence $2 \deg^-(v) > 2 \deg^-(v)$, a
  contradiction.

  Let $c \in L'(v)$ be such that $p_c \leq 1$. By applying~\cref{alg:list} on
  $v$ for color $c$ in $\alpha$ (resp. $\beta$), we obtain a coloring $\alpha'$
  (resp. $\beta'$). We have $\alpha'(w) = \alpha(w)=\beta(w)=\beta'(w)$ for all
  $w \succ v$ and $\alpha'(v)=c=\beta'(v)$, hence $\alpha'(w) = \beta'(w)$ for
  all $w \succcurlyeq v$. Furthermore, at most $O(n)$ Kempe changes are
  performed in the first setting, and at most $O(1)$ in the second. We obtain
  the claimed bounds on the length of the sequence of Kempe changes.
\end{proof}

\section{Recoloring with graphs of bounded degree}\label{sec:delta}
We recall~\cref{thm:delta}: \thmdelta*

To prove~\cref{thm:delta}, we adapt the
proof~\cite{bonamy2019conjecture,feghali2017kempe} by handling separate cases
depending on whether $G$ is 3-connected. The problem is then reduced to
$(\Delta(G)-1)$-degenerate graphs, with all vertices but possibly one of degree
at most $\Delta(G)$. We proved in~\cref{sec:blackbox} that in that case,
$R^k(G)$ has diameter $O(n^2)$ whenever $k \ge \Delta(G)$
(see~\cref{prop:blackbox}).

\subsection{Naive bounds for Mohar's conjecture}

Let $G$ be a graph of maximum degree $\Delta$ other than the 3-prism and let $k \ge
\Delta$. Bonamy \emph{et. al.}~\cite{bonamy2019conjecture} and Feghali
\emph{et. al.}~\cite{feghali2017kempe} proved that $\mathcal{C}^k(G)$ forms a
single Kempe class. If $k >
\Delta$ or if $G$ is not regular,~\cref{prop:blackbox} states that $\diam(R^k(G))
= O(n^2)$. Assume that $k = \Delta$ and that $G$ is regular. Let $u$ be a
vertex of $G$. Given any $\Delta$-coloring of $G$, there are at least two
neighbors of $u$ that are colored alike. Denote $G_{v+w}$ the graph where two
non-adjacent neighbors $v$ and $w$ of $u$ are identified and
$\mathcal{C}^\Delta_{v,w}(G)$ the set of $k$-colorings of $G$ in which $v$ and
$w$ are colored alike. We have
$$ \mathcal{C}^\Delta(G) = \mathop{\bigcup_{v,w \in N(u)}}_{vw \notin E(G)}
\mathcal{C}_{v,w}^\Delta(G).$$

\begin{figure}[h!]
  \begin{subfigure}[b]{.5\linewidth}
    \centering
    \begin{tikzpicture}
      \node[draw=black,fill=orange,circle,inner sep=3pt]   (A) at (180:2cm) {};
      \node[left=.1] (u) at (180:2cm) {$u$};
      \node[draw=black,fill=magenta,circle,inner sep=3pt]  (B) at (120:2cm) {};
      \node[draw=black,fill=orange,circle,inner sep=3pt] (C) at (60:2cm) {};
      \node[draw=black,fill=cyan,circle,inner sep=3pt]   (D) at (0:2cm) {};
      \node[draw=black,fill=orange,circle,inner sep=3pt]  (E) at (300:2cm) {};
      \node[draw=black,fill=blue,circle,inner sep=3pt]  (F) at (240:2cm) {};

      \node[draw=black,fill=cyan,circle,inner sep=3pt]   (1) at (135:1cm) {};
      \node[right=.1]  (v) at (135:1cm) {$v$};
      \node[draw=black,fill=cyan,circle,inner sep=3pt]  (4) at (225:1cm) {};
      \node[right=.1]  (w) at (225:1cm) {$w$};
      \node[draw=black,fill=magenta,circle,inner sep=3pt] (3) at (315:1cm) {};
      \node[draw=black,fill=blue,circle,inner sep=3pt]   (2) at (45:1cm) {};
      \draw (A) -- (B) -- (C) -- (D) -- (E) -- (F) -- (A) -- (1) -- (C) -- (2)
      -- (B) -- (1) -- (3) -- (D) -- (2) -- (4) -- (E) -- (3) -- (F) -- (4) -- (A);
    \end{tikzpicture}
  \end{subfigure}%
  \begin{subfigure}[b]{.5\linewidth}
    \centering
    \begin{tikzpicture}
        \node[draw=black,fill=orange,circle,inner sep=3pt]   (A) at (180:2cm) {};
        \node[left=.1] (u) at (180:2cm) {$u$};
        \node[draw=black,fill=magenta,circle,inner sep=3pt]  (B) at (120:2cm) {};
        \node[draw=black,fill=orange,circle,inner sep=3pt] (C) at (60:2cm) {};
        \node[draw=black,fill=cyan,circle,inner sep=3pt]   (D) at (0:2cm) {};
        \node[draw=black,fill=orange,circle,inner sep=3pt]  (E) at (300:2cm) {};
        \node[draw=black,fill=blue,circle,inner sep=3pt]  (F) at (240:2cm) {};

        \node[draw=black,fill=cyan,circle,inner sep=3pt]   (1) at (180:1cm) {};
        \node[right=.1]  (v) at (180:1cm) {\ \  ${v+w}$};
        \node[draw=black,fill=magenta,circle,inner sep=3pt] (3) at (315:1cm) {};
        \node[draw=black,fill=blue,circle,inner sep=3pt]   (2) at (45:1cm) {};

        \draw (A) -- (B) -- (C) -- (D) -- (E) -- (F) -- (A) -- (1) -- (C) -- (2)
        -- (B) -- (1) -- (3) -- (D) -- (2) -- (1) -- (E) -- (3) -- (F) -- (1) -- (A);
      \end{tikzpicture}
  \end{subfigure}%
  \caption{The graph $G_{v+w}$ (right) obtained from the 4-regular graph $G$
    (right) is 3-degenerate with all its vertices but $v+w$ of degree less than
    4\label{fig:contraction}}
\end{figure}
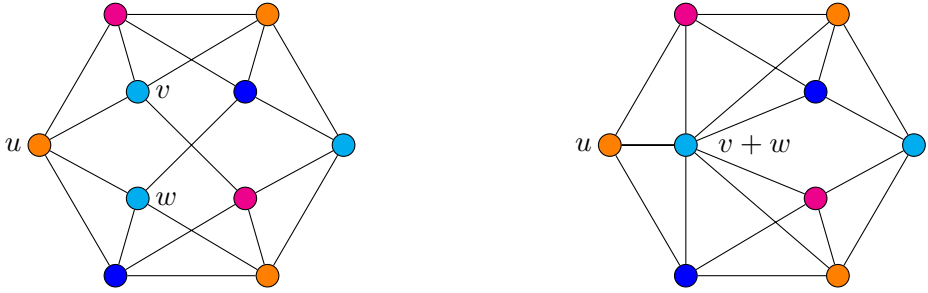

There is a one-to-one correspondence between the colorings of $G_{v+w}$ and the
colorings of $G$ in which $v$ and $w$ are colored alike. Performing a \Kcg in
$G_{v+w}$ corresponds to performing one or two \Kcgs in $G$, to maintain
identical colors on $v$ and $w$. As a result,
\begin{equation}\label{eq:identification}
  \diam(R^\Delta(G)) \ \le \ 2
  \mathop{\sum_{v,w \in N(u)}}_{vw \notin E(G)} \diam(R^\Delta(G_{v+w})).
\end{equation}
Note that $G_{v+w}$ is $(\Delta-1)$-degenerate and all its vertices but the
identification of $v$ and $w$ are of degree less than $\Delta$
(see~\cref{fig:contraction}). By~\cref{prop:blackbox}, $R^\Delta(G_{v+w})$ has
diameter at most $O(n^2)$. Together with~\cref{eq:identification}, this
proves that $\diam(R^\Delta(G)) = O(\Delta^2n^2)$. As a result, we obtain the
following theorem :

\begin{theorem}\label{thm:cubic_bound}
  Let $G$ be a graph of maximum degree $\Delta$, different from the 3-prism and
  $k \ge \Delta$. Then $R^k(G)$ has diameter $O(\Delta^2n^2)$.
\end{theorem}

To obtain~\cref{thm:delta}, one needs to improve this bound to $O(n^2)$. As
explained before, we only need to consider the case $k = \Delta$. Without loss
of generality, we may only consider the case $k \ge 4$, by including the
$\Delta^2$ obtained for $k = 3$ inside the $O$. The proof is very similar to the
one developed in~\cite{bonamy2019conjecture}, where we simply replace the use of
the Las Vergnas and Meyniel's lemma~\cite{lasvergnas1981kempe}
with~\cref{prop:blackbox} and count the number of Kempe changes performed
(see~\cref{app} for an adaptation of this proof).

\section{Recoloring graphs of bounded maximum average degree}\label{sec:mad}
A $t$-\emph{layering} of a graph $G$ is an ordered partition $V = V_1 \sqcup
\ldots \sqcup V_t$ of its vertices into $t$ subsets. We call \emph{layers} the
atoms $V_i$ of the partition.  Given a $t$-layering, we denote $G_i =
G[\bigcup_{j \ge i} V_j]$ for $1 \le i \le t$ and we define the \emph{level}
$\ell(v)$ of any vertex $v$ as the index $i$ of the subset $V_i$ it belongs
to. We say that a $t$-layering has \emph{degree} $k$ if $v$ has degree at most
$k$ in $G_{\ell(v)}$ for all $v \in V$.

We recall~\cref{thm:mad}:
\thmmad

The proof of~\cref{thm:mad} can be decomposed in two main steps. First, proving
that if $G$ has $\mad$ less than $k - \epsilon$, then it admits a $t$-layering
of degree $k-1$ with $t$ being logarithmic in $n$. This is achieved
by~\cref{prop:loglayers}, see~\cite{bousquet2015sparse}. Second, proving that if
$\alpha_i$ and $\beta_i$ are two colorings of $G[V_i]$ that differ by only one
Kempe change, then any extension $\alpha$ of $\alpha_i$ to $G$ differs by at
most $O(Poly_\varepsilon(n))$ kempe changes from an extension $\beta$ of
$\beta_i$ to $G$, see~\cref{prop:alg:recolor} applied with $t$ as
above. Finally, by the list coloring version of~\cref{prop:blackbox}, each layer
$G[V_i]$ can be be recolored one by one with $O(|V_i|)$ Kempe changes, starting
from $G[V_t]$ down to $G[V_1]$.

\begin{proposition}[\cite{bousquet2015sparse}]\label{prop:loglayers}
  For every $k \ge 1$ and every $\varepsilon > 0$, there exists a constant $C =
  C(k, \varepsilon) > 0$ such that every graph $G$ on $n$ vertices that
  satisfies $\mad(G) \le k - \varepsilon$ admits a $(C\log_kn)$-layering
  of degree $k-1$.
\end{proposition}

Consider a graph $G$ and a $t$-layering of degree $k-1$ of it. Consider an
arbitrary total order $\prec$ on the vertices that satisfies:
$$\forall i < j, \forall (u,v) \in V_i\times V_j, u \prec v.$$ Note that every
vertex has at most $k-1$ greater neighbors. We say a sequence $S_1$ of vertices
is \emph{lexicographically smaller} than another sequence $S_2$ if
\begin{itemize}
\item $S_2$ is empty and $S_1$ is not, or
\item the first vertex of $S_1$ is smaller than the first vertex of $S_2$, or
\item $S_1=x\oplus S_1'$ and $S_2=x\oplus S_2'$ for some vertex $x$, and $S_1'$ is
  lexicographically smaller than $S_2'$.
\end{itemize}
We then denote $S_1 \prec_{lex} S_2$. Note that in particular, the empty
sequence is the biggest element for this order. A sequence of vertices $S= (v_1,
\ldots, v_p)$ is said to be \emph{level-decreasing} if the sequence of levels
$(\ell(v_1), \ldots, \ell(v_p))$ is decreasing. We will say that two colorings
\emph{agree} on a set of vertices $X$ if their restrictions to $X$ are equal.

Given a coloring $\alpha$, a vertex $v$ and a color $c$, we say in this
subsection that a vertex $u$ is \emph{problematic} for the pair $(v,c)$ if there
exists a level-decreasing path of vertices in $K_{v,c}(\alpha,G)$ going from $v$
to $u$, such that $u$ has at least two neighbors in
$K_{v,c}(\alpha,G_{\ell(v)})$. Note that if the pair $(v,c)$ has no problematic
vertices and $c$ is not used in $\alpha(N_{G_{\ell(v)}}[v])$, then the coloring
$\beta$ resulting from the Kempe change $K_{v,c}(\alpha,G)$ agrees with $\alpha$
on $V(G_{\ell(v)}) \setminus \{v\}$ and has $\beta(u) = c$.

\begin{proposition}\label{prop:alg:recolor}
  Let $\alpha$ be a $k$-coloring of $G$ and $K$ a Kempe chain in $G_i$ with
  $\alpha_{|G_i}$. Let $\gamma$ be the coloring of $G_i$ obtained from
  $\alpha_{|G_i}$ by performing the Kempe change on $K$ in $G_i$.
  
  There exists a $k$-coloring $\beta$ of $G$ within $n^2 \cdot (2k)^t$ Kempe
  changes of $\alpha$ such that $\beta_{|G_i} = \gamma$.
\end{proposition}

\subsection{Freeing one color at a single vertex}

\begin{lemma}\label{lem:free}
  Let $\alpha$ be a $k$-coloring of $G$ and $v$ be a vertex of $G$. For any
  color $c$, ~\cref{alg:free} on input $(\alpha, (v), c)$ yields a $k$-coloring
  $\beta$ of $G$ within $n(2(k-1))^t$ Kempe changes of $\alpha$, such that
  $\alpha$ agrees with $\beta$ on $G_{\ell(v)}$, and the pair $(v,c)$ admits no
  problematic vertices for $\beta$.
\end{lemma}

\begin{algorithm}[h!]
  \SetAlgoLined \SetKwInOut{Input}{Input} \SetKwInOut{Output}{Output}
  \caption{}
  \label{alg:free}
  \Input{A $k$-coloring $\alpha$ of $G$, a level-decreasing
    sequence of vertices $S$, a color $c$.}
  \Output{A $k$-coloring $\beta$ of $G$ which agrees with $\alpha$ on
    $V(G_{\ell(v)})$ where $v$ is the last vertex of $S$. Moreover, the pair
    $(v,c)$ admits no problematic vertices with respect to $\beta$.}
  Let $v$ be the last vertex of $S$\;
  Let $\beta = \alpha$\;
  \While{the pair $(v,c)$ admits problematic vertices with respect to the
    coloring $\beta$}
        { Let $u$ be the largest problematic vertex for $(v,c)$ with respect to
          $\beta$\;
          \nl\label{alg:free:fresh_color}Let $c_u$ be a color in $[k] \setminus \beta(N_{G_{\ell(u)}}[u])$\;
          Let $\beta$ be the result of~\cref{alg:free} on input $(\beta, S
          \oplus u, c_u)$\;
          \nl\label{alg:free:change} Perform $K_{u,c_u}(\beta,G)$ in $\beta$\;
        }
  Return $\beta$\;
\end{algorithm}

\begin{proof}[of correctness of~\cref{alg:free}]
  We prove correctness of the algorithm by induction on $\ell(v)$. If $\ell(v) =
  1$, then $V(G_{\ell(v)}) = V$, so $(v,c)$ does not admit any problematic
  vertex. For $\ell(v) > 1$, at each iteration of the while loop,
  line~\ref{alg:free:fresh_color} is possible since $u$ has at least two
  neighbors in $G_{\ell(u)}$ that are colored identically. By induction
  hypothesis, the Kempe change line~\ref{alg:free:change} does not modify the
  color of the vertices in $G_{\ell(u)} \setminus \{u\}$. Thus, $\ell(v)$
  decreases at each iteration of the loop and at most $n$ calls are generated by
  the current call.
\end{proof}

To conclude the proof of~\cref{lem:free}, we need to bound the number of Kempe
changes performed by~\cref{alg:free}.

Given a call $C$ of~\cref{alg:free}, we will denote $S_C$ the sequence provided
in input, and $u_C$ its last vertex.
\begin{observation}\label{obs:path}
  If~\cref{alg:free} is called on $(\alpha, S, c)$ where $S$ is a
  level-decreasing sequence, and makes some recursive call $C$, then the
  sequence $S_C$ is longer than $S$ and is of the form $S \oplus u_C$. Moreover,
  $S_C$ is also a level-decreasing sequence by construction.
\end{observation}

\begin{claim}[analog to \cite{bousquet2015sparse}]\label{cl:order_seq}
  If a call $D$ is initiated after a call $C$, then $S_D \prec_{lex} S_C$.
\end{claim}
\begin{proof}
  If $D$ is called by $C$, then by~\cref{obs:path}, $S_D$ is of the form $S_C
  \oplus u_C$ thus $S_D \prec_{lex} S_C$. By applying this argument inductively,
  we have also have $S_D \prec_{lex} S_C$ if the call $D$ is generated by
  $C$.\\ Now, assume that $D$ is not generated by $C$. Denote $I$ the initial
  call of~\cref{alg:free} and recall that the recursive calls generated by $I$
  have a natural tree structure, rooted at $I$. There exists a unique sequence
  $\mathcal{S}_C = C_1, C_2, \dots C_{t_1}$ such that $C_i$ calls $C_{i+1}$ for
  each $i <t_1$, with $C_1 = I$ and $C_{t_1} = C$; and a unique sequence
  $\mathcal{S}_D = D_1, D_2, \dots D_{t_2}$ such that $D_i$ calls $D_{i+1}$ for
  each $i <t_2$, with $D_1 = I$ and $D_{t_2} = D$. Denote $B$ the last common
  ancestor of $C$ and $D$. Since neither $C$ nor $D$ is generated by the other,
  $B$ is distinct from $C$ and $D$ and thus is not the last element of the
  sequences $\mathcal{S}_C$ and $\mathcal{S}_D$. We have :
  \begin{itemize}
  \item $B_C$ and $B_D$ are both called by $B$ and $B_C$ is initiated before $B_D$,
  \item $B_C = C$ or $B_C$ generates $C$,
  \item $B_D = D$ or $B_D$ generates $D$.
  \end{itemize}
  It follows from~\cref{obs:path} that $S_{B_C} \prec_{lex} S_C$, so we just need to
  show that $S_D \prec_{lex} S_{B_C}$. By~\cref{obs:path} applied inductively,
  $S_D$ can be written as $S_{B_D} \oplus T_D$ with $S_{B_D} = S_B
  \oplus u_{B_D}$. Furthermore, $S_{B_C}$ can be written as $S_B \oplus
  u_{B_C}$. Since $B_C$ is called before $B_D$, we have $u_{B_D} \prec u_{B_C}$, so
  $S_D = S_B \oplus u_{B_D} \oplus T_D \prec_{lex} S_B \oplus u_{B_C} = S_{B_C}$.
\end{proof}

\begin{claim}[\cite{bousquet2015sparse}]\label{cl:nb_path} Given that the $t$-layering $V = V_1
  \sqcup \dots \sqcup V_t$ has degree at most $k-1$, the number of
  level-decreasing paths between two vertices $u$ and $w$ in different levels is
  at most $(k-1)^{i-1}$ where $i = |\ell(u) - \ell(v)|$.
\end{claim}

\begin{lemma}\label{lem:double}
  The number of Kempe changes performed by
  \cref{alg:free} on input $(\alpha, (v), c)$ is bounded by $n(2(k-1))^t$.
\end{lemma}
\begin{proof}
  Observe that the sequences of vertices considered in the recursive calls are
  subsequences of level-decreasing paths. Thus, each level-decreasing path
  between $v$ and $w$ has at most $2^{\ell(v) - \ell(w) -1}$ subsequences that
  contain $v$ and $w$. As a result, each vertex $w$ with level $\ell(w)$ less
  than $\ell(v)$ is the origin of at most $2^{\ell(v) - \ell(w) -1}(k-1)^{\ell(v)
    -\ell(w) -1}$ Kempe changes.
\end{proof}

\subsection{Performing a Kempe change}
\begin{algorithm}[h!]
  \SetAlgoLined \SetKwInOut{Input}{Input} \SetKwInOut{Output}{Output}
  \caption{}
  \label{alg:recolor}
  \Input{A $k$-coloring $\alpha$ of $G$ with $k \ge \ell+1$, a level $i$, a
    Kempe chain $K$ of $G[V_i]i$.}
  \Output{A $k$-coloring $\beta$ of $G$, whose restriction to $G[V_i]$ is the
    coloring resulting from the Kempe change $K$ in $\alpha_{|G[V_i]}$.}
  Let $c_1,c_2$ be the colors involved in $K$\;
  Let $\beta = \alpha$\;
  \While{there exists a vertex $v \in K$ such that $(v,c_1)$ or $(v,c_2)$ admits problematic vertices in $\beta$}
        { Among all problematic vertices for some $(v,c_1)$ or $(v,c_2)$ with $v \in K$, let $u$ be the largest one\;
          Let $v \in K$ and $j \in \{1,2\}$ be such that $u$ is problematic fr $(v,c_j)$ in $\beta$\;
          Update $\beta$ with the result of~\cref{alg:free} on input $(\beta, (v), c_j)$.
        }
  Perform the Kempe change $K$ on $\beta$\;
  Return $\beta$\;
\end{algorithm}

\begin{proof}[of~\cref{prop:alg:recolor}]
  The coloring $\beta$ is obtained by applying~\cref{alg:recolor}. The
  correctness of \cref{alg:recolor} follows from the correctness of
  \cref{alg:free:change}, which holds by \cref{lem:free}.  After each iteration
  of the loop the largest possible vertex that is problematic for a pair of
  vertex and a color of $K$ decreases. Therefore, the loop is executed at most
  $n$ times and at the end, the Kempe chain $K$ is a proper Kempe chain of
  $G$. Combined with the cost analysis of \cref{alg:free:change} by
  \cref{lem:double}, this proves that~\cref{alg:recolor} performs at most
  $n^2(2(k-1))^i+1$ Kempe changes and is correct.
\end{proof}

\subsection{Combining the arguments}

We are now ready to prove~\cref{thm:mad}.

\begin{proof}[of~\cref{thm:mad}]
  Let $V_1 \sqcup \ldots V_t$ be a $t$-layering of $G$ of degree $(k-1)$, with $t =
  C\log_kn$ (see~\cref{prop:loglayers}).

  We claim that $G$ can be recolored layer by layer, starting from $G[V_t]$, with
  a polynomial number of Kempe changes. We prove this by decreasing induction on
  the level of the layer. Let $1 \le i \le t$, let $\alpha$ and $\beta$ be two
  $k$-colorings of $G$ and assume that $\alpha$ and $\beta$ agree on all
  vertices of level more than $i$. For $v \in V_i$, let $L(v) = [k] \setminus
  \alpha(N(v) \cap G_{i+1})$. We have for all $v \in V_i$, $|L(v)| \ge
  \deg_{G[V_i]}(v) +1$, so by applying the second case of~\cref{prop:blackbox},
  there exist a sequence $S$ of Kempe changes in $G[V_i]$ of size $O(|V_i|)$
  leading from $\alpha_{|V_i}$ to $\beta_{|V_i}$. Moreover, each of the Kempe
  change in $S$ is a proper Kempe change in $G_i$. By~\cref{prop:alg:recolor},
  each of these Kempe changes can be performed in $G$, after
  $O(Poly_\varepsilon(n))$ Kempe changes affecting the vertices of level less
  than $i$.
\end{proof}

\section{Recoloring bounded treewidth graphs}\label{sec:tw}
Let $G$ be an $n$-vertex graph of treewidth $\tw$.  Let $H$ be a chordal graph
such that $G$ is a (not necessarily induced) subgraph of $H$, with $\omega(H) =
\chi(H) \le \tw+1$ and $V(H)=V(G)$. Computing $H$ is equivalent to computing a
so-called tree decomposition of $G$, which can be done in time $f(\tw)\cdot
n$~\cite{bodlaender1996linear}.

Since $G$ and $H$ are defined on the same vertex set, there may be confusion
when discussing neighbourhoods and other notions. When useful, we write $G$ or
$H$ in index to specify. There is an ordering $v_1 \prec \ldots \prec v_n$ of
the vertices of $G$ such that for all $v \in V(G), N^+_H[v]$ induces a clique in
$H$. The ordering can be computed from $H$ in $O(n)$ using
Lex-BFS~\cite{rose1975elimination}.

The core of the proof lies in Proposition~\ref{prop:tw}: for $k \ge \tw +1$, any
$k$-coloring of $G$ is equivalent up to $O(\tw \cdot n^2)$ Kempe changes to a
$k$-coloring of $G$ that yields a $k$-coloring of $H$.

\begin{proof}[of~\cref{thm:tw} assuming~\cref{prop:tw}]
Let $\alpha$ and $\beta$ be two $k$-colorings of $G$.  By
Proposition~\ref{prop:tw}, there exists a $k$-coloring $\alpha'$
(resp. $\beta'$) that is equivalent to $\alpha$ (resp. $\beta$) up to $O(\tw
\cdot n^2)$ \Kcgs. Additionally, both $\alpha'$ and $\beta'$ yield $k$-colorings
of $H$. Since $H$ is chordal, by Proposition~\ref{prop:chordal}, there exists a
sequence of at most $n$ \Kcgs in $H$ from $\alpha'$ to $\beta'$. Each of these
\Kcgs in $H$ can be simulated by at most $n$ \Kcgs in $G$, which results in a
sequence of length $O(\tw \cdot n^2)$ between $\alpha$ and $\beta$.
\end{proof}

\begin{proposition}
  \label{prop:tw}
  Given any $k$-coloring $\alpha$ of $G$ with $k \ge \tw +1$, there exists a
  $k$-coloring $\alpha'$ of $G$ that is equivalent to $\alpha$ up to $O(\tw
  \cdot n^2)$ \Kcgs and such that $\alpha'(u) \neq \alpha'(v)$ for all $uv \in
  E(H)$. The algorithm \ref{alg:tw} computes $\beta$ and a sequence of \Kcgs
  leading to it.
\end{proposition}

To prove Proposition~\ref{prop:tw} and obtain a $k$-coloring of $H$, we
gradually ``add'' to $G$ the edges in $E(H) \setminus E(G)$. To add an edge, we
first reach a $k$-coloring where the extremities have distinct colors, then
propagate any later Kempe change involving one extremity to the other
extremity. We formalize this process through Algorithm~\ref{alg:tw}. Let $v_1w_1
\prec \ldots \prec v_qw_q$ be the edges in $E(H) \setminus E(G)$ in the
lexicographic order, where $v_i \prec w_i$ for every $i$.

\begin{algorithm}[h!]
  \SetAlgoLined \SetKwInOut{Input}{Input} \SetKwInOut{Output}{Output}
  \caption{Going from $\alpha$ to $\alpha'$}
  \label{alg:tw}
  \Input{$G$ a graph of treewidth $\tw$, a $k$-coloring $\alpha$ of $G$ with $k
    \ge \tw +1$ and $v_1w_1 \prec \ldots \prec
    v_qw_q$ such that $G+\{v_1w_1,\ldots,v_qw_q\}$ is a $\tw + 1$-colorable
    chordal graph $H$}
  \Output{A $k$-coloring $\widetilde{\alpha}$ of $H$, that
    is Kempe equivalent to $\alpha$}
  Let $\widetilde{G} \leftarrow G$ and $\widetilde{\alpha} \leftarrow \alpha$\;
  \nl\label{alg:boucle1}\For{$j$ from $1$ to $q$}
           { \If{$\widetilde{\alpha}(v_j) = \widetilde{\alpha}(w_j)$}
             { Let $c \in [k]\setminus \widetilde{\alpha}(N^+_H[v_j])$\;
               \tcp{Possible because $\widetilde{\alpha}(v_j) =
                 \widetilde{\alpha}(w_j)$ and $|N^+_H[v_j]|\leq \tw \le k-1$.}
               Let $U = N^-_{\widetilde{G}}(v_j) \cap N^-_{\widetilde{G}}(w_j) =
               \{ u_1 \prec \ldots \prec u_p\}$\;
               \nl\label{alg:boucle2}\For{$i = p$ down to 1}
                        { \If{$\widetilde{\alpha}(u_i) = c$}
                          { Let $c_i \in [k]\setminus
                            \widetilde{\alpha}(N^+_{\widetilde{G}}[u_i])$\;
                            \tcp{Possible because $\widetilde{\alpha}(v_j) =
                              \widetilde{\alpha}(w_j)$ and
                              $|N^+_{\widetilde{G}}[u_i]|\leq \tw \le k-1$.}
                            \nl\label{alg:line:1} $\widetilde{\alpha}\leftarrow
                            K_{u_i,c_i}(\widetilde{\alpha},\widetilde{G})$\;
                          }
                          \nl\label{assert1} \tcp{Now $c \notin \widetilde{\alpha}(\{x \in
                            N_{\widetilde{G}}(v_j) | x \ge u_i\})$}
                        }
                        \nl\label{assert2} \tcp{Now $c \notin
                          \widetilde{\alpha}(N_{\widetilde{G}}(v_j))$}
                        \nl\label{alg:line:2}$\widetilde{\alpha}\leftarrow
                        K_{v_j,c}(\widetilde{\alpha},\widetilde{G})$\;
             }
             $\widetilde{G} \leftarrow \widetilde{G} \cup \{v_jw_j\}$\;
           }
\end{algorithm}

We will prove the following three claims. Note that Proposition~\ref{prop:tw}
follows from Claims 1 and 2, while Claim 3 simply guarantees that the proof
of~\cref{thm:tw} is indeed constructive.

\begin{claim}[1]\label{cl:algocorrect}
Algorithm~\ref{alg:tw} outputs a $k$-coloring $\alpha'$ of $G$ that is Kempe
equivalent to $\alpha$ and such that $\alpha'(u) \neq \alpha'(v)$ for all $uv
\in E(H)$.
\end{claim}

\begin{claim}[2]\label{cl:length}
Algorithm~\ref{alg:tw} performs $O(\tw \cdot n^2)$ Kempe changes in $G$ to obtain
$\alpha'$ from $\alpha$.
\end{claim}

\begin{claim}[3]\label{cl:algocomp}
Algorithm~\ref{alg:tw} runs in $O(\tw \cdot n^4)$ time.
\end{claim}

In Algorithm~\ref{alg:tw}, the variable $\widetilde{G}$ keeps track of how close
we are to a $k$-coloring of $H$. Before the computations start,
$\widetilde{G}=G$. When the algorithm terminates, $\widetilde{G} = H$. At every
step, $G$ is a subgraph of $\widetilde{G}$.  To refer to $\widetilde{G}$ or
$\widetilde{\alpha}$ at some step of the algorithm, we may say the \emph{current}
graph or \emph{current} coloring.
The \Kcgs that we discuss are performed in $\widetilde{G}$. Consequently, the
corresponding set of vertices might be disconnected in $G$, and every \Kcg in
$\widetilde{G}$ may correspond to between $1$ and $n$ \Kcgs in $G$.

\begin{proof}[of Claim 1]
By construction, at every step $\widetilde{\alpha}$ is Kempe equivalent to $\alpha$.
We prove the following loop invariant: at every step, $\widetilde{\alpha}$ is a
$k$-coloring of $\widetilde{G}$. Since $\widetilde{G} =H$ at the end of the
algorithm, proving the loop invariant will yield the desired conclusion.

The invariant holds at the beginning of the algorithm, when $\widetilde{G} =
G$.

Assume that at the beginning of the $j$-th iteration of the
loop~\ref{alg:boucle1}, $\widetilde{\alpha}$ is a proper coloring of
$\widetilde{G}$. All the \Kcgs in the loop are performed in $\widetilde{G}$,
so we only need to prove that at the end of the iteration,
$\widetilde{\alpha}(v_j) \neq \widetilde{\alpha}(w_j)$.

This follows from the validity of comments~\ref{assert1}
and~\ref{assert2}. The latter is a direct consequence of the former, so we
focus on arguing that after the step $i$ of the inner loop~\ref{alg:boucle2},
we have $c \notin \widetilde{\alpha}(\{x \in N_{\widetilde{G}}(v_j) | x \ge
u_i\})$.
  The key observation is that at the $i$-th step of the inner
  loop~\ref{alg:boucle2}, the \Kcgs performed at line~\ref{alg:line:1} involve
  only vertices smaller than $u_i$.  We prove by induction the stronger
  statement the \Kcn $T$ involved in the \Kcg $K_{u_i,c_i}(\widetilde{\alpha},
  \widetilde{G})$ is a tree rooted at $u_i$ in which all the nodes are smaller
  than their father.
  \begin{itemize}
  \item $c_i \not\in \widetilde{\alpha}(N_{\widetilde{G}}^+(u_i))$ so all the
    vertices at distance 1 in $T$ from $u_i$ are smaller than $u_i$.
  \item Let $x$ at distance $d+1$ from $u_i$ in $T$. Let $y$ be a neighbour of
    $x$ at distance $d$ from $u_i$. Assume by contradiction that $x \succ y$. By
    induction hypothesis, there is a unique neighbour $z$ of $y$ at distance
    $d-1$ from $u_i$, with $y \prec z$. Both $z$ and $x$ are in $N^+_H(y)$ and
    since $H$ is chordal, this implies $zx \in E(H)$.
    We have $z \prec u_i \prec v_j$ so $zx \prec v_jw_j$ and $zx \in
    E(\widetilde{G})$. In particular, $x$ is at distance $d$ from $u_i$ in $T$,
    which raises a contradiction and proves $x \prec y$.\\ Assume by
    contradiction that $x$ is adjacent to two vertices $y$, $z$ at distance $d$
    from $u_i$ in $T$. Then $y$ and $z$ are identically colored so $yz \notin
    E(\widetilde{G})$. Moreover $y,z \in N^+_H(x)$ and $H$ is chordal, hence
    $yz$ is an edge of $H$. Since $y,z \prec v_j$, we have $yz \prec
    v_jw_j$. Thus, $yz$ belongs to $\widetilde{G}$, raising a contradiction.
  \end{itemize}
  As a result, the \Kcg of line~\ref{alg:line:2} does not recolor any vertex
  larger than $u_i$ with color $c$, and the comment~\ref{assert2} is
  true. Therefore at the beginning of line~\ref{alg:line:2},
  $\widetilde{\alpha}(v_j) = \widetilde{\alpha}(w_j)$ and $c \notin \alpha(N(v_j))$. At
  the end of line~\ref{alg:line:2}, $v_j$ and $w_j$ are colored
  differently.
\end{proof}

\begin{proof}[of Claim 2]
  We now prove that the number of \Kcgs in $G$ performed by the algorithm is
  $O(\tw \cdot n^2)$.

  We first prove that for each vertex $x$, there exists at most one step $j$ of
  the loop~\ref{alg:boucle1} for which $v_j=x$ and we enter the conditional
  statement $\widetilde{\alpha}(v_j) = \widetilde{\alpha}(w_j)$. Indeed, the first time
  we enter the conditional statement, the vertex $x$ is recolored with a color
  $c$ not in $N^+_H(x)$ at line~\ref{alg:line:2}. Note that all the edges $xy$
  with $y \succ x$ are consecutive in the ordering of $E(H)\setminus
  E(G)$. Therefore, once the vertex $x$ is recolored, all the remaining edges
  $xy$ are handled without Kempe change, as the conditional statement is not
  satisfied.  This implies directly that line~\ref{alg:line:2} is executed at
  most $n$ times.

  Now, we bound the number of times $x$ plays the role of $u_i$ in the \Kcg at
  line~\ref{alg:line:1}. For each step $j$ of loop~\ref{alg:boucle1} for which
  it happens, we have $v_j, w_j \in N_H^+(x)$. Since $|N^+_H(x)| \leq \tw$ and
  each $v_j$ is involved at most once by the above argument, we obtain that $x$
  plays this role at most $\tw$ times.

  Consequently the overall number of \Kcgs performed in $\widetilde{G}$ by the
  algorithm is $O(\tw\cdot n)$. Performing a \Kcg in $\widetilde{G}$ is equivalent to
  performing a \Kcg in all the connected component of $G$ of the \Kcn of
  $\widetilde{G}$. Therefore, the number of \Kcgs performed in $G$ by the
  algorithm is $O(\tw\cdot n^2)$.
\end{proof}

\begin{proof}[of Claim 3]
  In total, the loop~\ref{alg:boucle2} is executed at most once for every pair
  of vertices in $N^+_H(u)$ for each $u \in V(G)$, that is $O(\tw^2n)$
  times. However, we also take into account the number of \Kcgs that need to be
  performed. By Claim 2, only $O(\tw \cdot n^2)$ \Kcgs are performed in $G$. As a result,
  the total complexity of the algorithm is $O(\tw^2n + \tw\cdot n^4) = O(\tw\cdot n^4)$
  (performing \Kcgs in a naive way in $G$).
\end{proof}

\section*{Acknowledgments}
The authors would like to thank Carl Feghali for fruitful discussions.

\bibliographystyle{IEEEtranSA}  
\bibliography{biblio}

\providecommand{\etalchar}[1]{$^{#1}$}
\begin{thebibliography}{JRTL75}
\providecommand{\url}[1]{#1}
\csname url@samestyle\endcsname
\providecommand{\newblock}{\relax}
\providecommand{\bibinfo}[2]{#2}
\providecommand{\BIBentrySTDinterwordspacing}{\spaceskip=0pt\relax}
\providecommand{\BIBentryALTinterwordstretchfactor}{4}
\providecommand{\BIBentryALTinterwordspacing}{\spaceskip=\fontdimen2\font plus
\BIBentryALTinterwordstretchfactor\fontdimen3\font minus
  \fontdimen4\font\relax}
\providecommand{\BIBforeignlanguage}[2]{{%
\expandafter\ifx\csname l@#1\endcsname\relax
\typeout{** WARNING: IEEEtranSA.bst: No hyphenation pattern has been}%
\typeout{** loaded for the language `#1'. Using the pattern for}%
\typeout{** the default language instead.}%
\else
\language=\csname l@#1\endcsname
\fi
#2}}
\providecommand{\BIBdecl}{\relax}
\BIBdecl

\bibitem[BB18]{bonamy2013recoloring}
M.~Bonamy and N.~Bousquet, ``Recoloring graphs via tree decompositions,''
  \emph{European Journal of Combinatorics}, vol.~69, pp. 200--213, 2018.

\bibitem[BBFJ19]{bonamy2019conjecture}
\BIBentryALTinterwordspacing
M.~Bonamy, N.~Bousquet, C.~Feghali, and M.~Johnson, ``On a conjecture of
  {M}ohar concerning {K}empe equivalence of regular graphs,'' \emph{J. Comb.
  Theory, Series B}, vol. 135, pp. 179--199, 2019. [Online]. Available:
  \url{https://www.sciencedirect.com/science/article/pii/S009589561830073X}
\BIBentrySTDinterwordspacing

\bibitem[BH19]{bousquet2019polynomial}
\BIBentryALTinterwordspacing
N.~Bousquet and M.~Heinrich, ``A polynomial version of cereceda's conjecture,''
  \emph{arXiv preprint arXiv:1903.05619}, 2019. [Online]. Available:
  \url{https://arxiv.org/abs/1903.05619}
\BIBentrySTDinterwordspacing

\bibitem[BHI{\etalchar{+}}20]{bonamy2020diameter}
M.~Bonamy, M.~Heinrich, T.~Ito, Y.~Kobayashi, H.~Mizuta, M.~Mühlenthaler,
  A.~Suzuki, and K.~Wasa, ``Diameter of colorings under {K}empe changes,''
  \emph{Theoretical Computer Science}, 2020.

\bibitem[Bod96]{bodlaender1996linear}
H.~L. Bodlaender, ``A linear-time algorithm for finding tree-decompositions of
  small treewidth,'' \emph{SIAM Journal on computing}, vol.~25, no.~6, pp.
  1305--1317, 1996.

\bibitem[BP16]{bousquet2015sparse}
N.~Bousquet and G.~Perarnau, ``Fast recoloring of sparse graphs,''
  \emph{European Journal of Combinatorics}, 2016.

\bibitem[{Cer}07]{cercedas2007mixing}
\BIBentryALTinterwordspacing
L.~{Cereceda}, ``\BIBforeignlanguage{english}{Mixing graph colourings},'' Ph.D.
  dissertation, The London School of Economics and Political Science, 12 2007,
  an optional note. [Online]. Available: \url{http://etheses.lse.ac.uk/131/}
\BIBentrySTDinterwordspacing

\bibitem[CR14]{carston2014brooks}
D.~Cranston and L.~Rabern, ``Brook's theorem and beyond,'' \emph{Journal of
  Graph Theory}, vol.~80, 2014.

\bibitem[FJP17]{feghali2017kempe}
\BIBentryALTinterwordspacing
C.~Feghali, M.~Johnson, and D.~Paulusma, ``\BIBforeignlanguage{english}{Kempe
  equivalence of colourings of cubic graphs},''
  \emph{\BIBforeignlanguage{english}{European journal of combinatorics}},
  vol.~59, pp. 1--10, 2017. [Online]. Available:
  \url{https://www.sciencedirect.com/science/article/pii/S0195669816300488}
\BIBentrySTDinterwordspacing

\bibitem[JRTL75]{rose1975elimination}
D.~J.~Rose, E.~R. Tarjan, and G.~S. Lueker, ``Algorithmic aspects of vertex
  elimination on graphs,'' \emph{SIAM Journal on Computing}, 1975.

\bibitem[Kem79]{kempe1879geographical}
A.~B. Kempe, ``On the geographical problem of the four colours,''
  \emph{American journal of mathematics}, vol.~2, no.~3, pp. 193--200, 1879.

\bibitem[LM81]{lasvergnas1981kempe}
\BIBentryALTinterwordspacing
M.~{Las Vergnas} and H.~{Meyniel}, ``\BIBforeignlanguage{english}{{Kempe
  classes and the Hadwiger conjecture}},''
  \emph{\BIBforeignlanguage{english}{{J. Comb. Theory, Ser. B}}}, vol.~31, pp.
  95--104, 1981. [Online]. Available:
  \url{https://doi.org/10.1016/S0095-8956(81)80014-7}
\BIBentrySTDinterwordspacing

\bibitem[Moh06]{mohar2006kempe}
B.~Mohar, ``\BIBforeignlanguage{english}{Kempe equivalence of colorings},'' in
  \emph{\BIBforeignlanguage{english}{Graph Theory in Paris}}.\hskip 1em plus
  0.5em minus 0.4em\relax Springer, 2006, pp. 287--297.

\bibitem[MS09]{mohar2009new}
B.~Mohar and J.~Salas, ``A new {K}empe invariant and the (non)-ergodicity of
  the {W}ang--{S}wendsen--{K}oteck{\`y} algorithm,'' \emph{Journal of Physics
  A: Mathematical and Theoretical}, vol.~42, no.~22, p. 225204, 2009.

\bibitem[Sok00]{sokal2000personal}
A.~D. Sokal, ``A personal list of unsolved problems concerning lattice gases
  and antiferromagnetic {P}otts models,'' \emph{arXiv preprint
  cond-mat/0004231}, 2000.

\bibitem[vdH13]{van2013complexity}
J.~van~den Heuvel, ``\BIBforeignlanguage{english}{The complexity of change.}''
  \emph{\BIBforeignlanguage{english}{Surveys in combinatorics}}, vol. 409, no.
  2013, pp. 127--160, 2013.

\bibitem[Vig00]{vigoda}
E.~Vigoda, ``\BIBforeignlanguage{english}{Improved bounds for sampling
  colorings},'' \emph{\BIBforeignlanguage{english}{Journal of Mathematical
  Physics}}, vol.~41, no.~3, pp. 1555--1569, 2000.

\bibitem[Viz64]{vizing1964estimate}
V.~G. Vizing, ``\BIBforeignlanguage{english}{On an estimate of the chromatic
  class of a p-graph},'' \emph{\BIBforeignlanguage{english}{Discret Analiz}},
  vol.~3, pp. 25--30, 1964.

\bibitem[Viz68]{vizing1968some}
------, ``\BIBforeignlanguage{english}{Some unsolved problems in graph
  theory},'' \emph{\BIBforeignlanguage{english}{Russian Mathematical Surveys}},
  vol.~23, no.~6, p. 125, 1968.

\end{thebibliography}

\clearpage
\appendix
\section{Proof of~\cref{thm:delta}}\label{app}
This section is an adaptation of the proof provided
in~\cite{bonamy2019conjecture} for $k \ge 4$. The proof handles separately the
cases of graphs not 3-connected, 3-connected graphs with diameter at most 2 and
3-connected graphs of diameter at least 3.

\subsection{$G$ not 3-connected}
We first need a few lemmas.
\begin{lemma}[Adapted from Lemma 3.2 in~\cite{lasvergnas1981kempe}]\label{lem:S_complete_permutation}
  Let $G_1$ and $G_2$ be two $(k-1)$-degenerate graphs of maximum degree $k$,
  such that $S = G_1 \cap G_2$ is a complete graph. Let $\alpha$ be a
  $k$-coloring of $G$ and $\beta_{|G_1}$ be a $k$-coloring of $G_1$. There
  exists a $k$-coloring $\gamma$ of $G$ within $O(|S| n^2)$ \Kcgs of $\alpha$
  such that $\gamma_{|G_1} = \beta_{|G_1}$ and $\gamma_{|G_2}$ is equal to
  $\alpha_{|G_2}$ up to a permutation of colors.
\end{lemma}

\begin{proof}
  By~\cref{prop:blackbox}, there exists a series of \Kcgs in $G_1$, leading
  from $\alpha_{|G_1}$ to $\beta_{|G_1}$, such that the color of each vertex of
  $G_1$ is modified at most $O(n_1)$ times. We will adapt this sequence to $G$ to
  ensure that at all times, the current coloring $\widetilde{\alpha}$ differs
  from $\alpha$ on $G_2$ only by a permutation of colors. For each \Kcg
  $K_{u,c}(\widetilde\alpha, G_1)$ in the sequence, if
  $K_{u,c}(\widetilde\alpha, G) \cap S = \emptyset$, then simply perform
  $K_{u,c}(\widetilde\alpha,G)$. If $K_{u,c}(\widetilde\alpha,G) \cap S \neq
  \emptyset$, then swap the colors $\widetilde\alpha(u)$ and $c$ in each \Kcn of
  $G_2$ in addition to performing $K_{u,c}(\widetilde\alpha,G)$.\\ The latter
  case occurs at most $O(|S| n_1)$ times, each time resulting in $O(n_2)$
  \Kcgs in $G_2$. The sequence we obtain has length $O(n_1^2+|S|n_1n_2)
  = O(|S|n^2)$.
\end{proof}

\begin{corollary}[Adapted from Lemma 3.3 in~\cite{lasvergnas1981kempe} and in~\cite{bonamy2019conjecture}]\label{cor:S_complete}
  Let $G_1$ and $G_2$ be two $(k-1)$-degenerate graphs of maximum degree $k$,
  such that $S = G_1 \cap G_2$ is a complete graph. Let $G = G_1 \cup G_2$, any
  two $k$-colorings of $G$ are equivalent up to $O(|S|n^2)$.
\end{corollary}

\begin{proof}
  Let $\alpha$ and $\beta$ be two $k$-colorings of
  $G$. By~\cref{lem:S_complete_permutation} applied to $\alpha$ and
  $\beta_{|G_1}$, there exists $\gamma$ within $O(|S| n^2)$ \Kcgs of
  $\alpha$, such that $\gamma_{|G_1} = \beta_{|G_1}$ and $\gamma_{|G_2}$ is
  equal to $\alpha_{|G_2}$ up to a permutation of
  colors. By~\cref{lem:S_complete_permutation} applied to $\gamma$ and
  $\beta_{|G_2}$, there exists $\delta$ within $O(|S| n^2)$ \Kcgs of
  $\gamma$, such that $\delta_{|G_2} = \beta_{|G_2}$ and $\delta_{|G_1}$ is
  equal to $\gamma_{|G_1} = \beta_{|G_1}$ up to a permutation of colors
  $\sigma$. Note that if a color $c$ is used in $S$, then $c = \sigma(c)$ and
  $\delta^{-1}(c) = \beta^{-1}(c)$, even on $G_1$. By iteratively swapping the
  other colors with their image by $\sigma$ in each \Kcn of $G_2$, one obtains
  the coloring $\beta$ within at most $O(kn)$ \Kcgs .
\end{proof}

\begin{proposition}[Adapted from Proposition 3.1 in~\cite{bonamy2019conjecture}]
  Let $k \ge 3$ and $G$ be a $k$-regular graph that is neither
  3-connected, nor a clique or the 3-prism. Any two $k$-colorings of $G$
  are equivalent up to $O(n^2)$ \Kcgs.
\end{proposition}

\begin{proof}
  Let $S$ be a separator of minimal size, $|S| = 1$ or $|S| = 2$. Consider $G_1$
  and $G_2$ such that $G_1 \cup G_2 = G$ and $G_1 \cap G_2 = S$. Two cases can
  occur:
  \begin{itemize}
  \item $S$ is a complete graph, then the result stems
    from~\cref{cor:S_complete}.
  \item $S$ is composed of two non-adjacent vertices $u$ and $v$. If both $u$
    and $v$ have only one neighbor in $G_1$, consider $w$ the unique neighbor of
    $u$ in $G_1$. Note that $w$ is not adjacent to $v$, otherwise, $\{w\}$ would be a
    separator of $G$. Since $G$ is $k$-regular with $k \ge 3$, $w$ has
    at least two neighbors in $G_1$. As a result, we can assume that $u$ or $v$
    has at least two neighbors in $G_1$ (respectively $G_2$). We prove the two
    following claims.
    \begin{claim}
      There is a sequence \Kcgs of length at most $O(n^2)$ between any
      two $k$-colorings $\alpha$ and $\beta$ of $G$, such that $\alpha(u)
      \neq \alpha(v)$ and $\beta(u) \neq \beta(v)$.
    \end{claim}
    Let $G_1'$ (respectively $G_2'$ and $G'$) the graphs obtained from $G_1$
    (respectively $G_2$ and $G$), by adding an edge between $u$ and $v$. The
    graph $G_1'$ (respectively $G_2'$) has maximum degree $k$ and is
    $(k-1)$-degenerate because either $u$ or $v$ has at most $k-2$
    neighbors in $G_1$ (respectively $G_2'$). By~\cref{cor:S_complete},
    $R^k(G')$ has diameter $O(n^2)$. As a result, $R^k(G)$ has
    also diameter $O(n^2)$.
    \begin{claim}\cite{bonamy2019conjecture}
      Given any $k$-coloring $\alpha$ of $G$, such that $\alpha(u) =
      \alpha(v)$, there exists a $k$-coloring $\beta$ within at most 3
      \Kcgs from $\alpha$ such that $\alpha'(u) \neq \alpha'(v)$.
    \end{claim}
    Assume that $\alpha(u) = \alpha(v) = c$. If there exists a color $c'$ that
    is unused in the closed neighborhood of $u$ or $v$, say $u$, then one obtains the
    desired property after the trivial \Kcg $K_{u,c'}(\alpha)$. Henceforth,
    assume that both $u$ and $v$ have a neighbor of each color. Two cases may
    occur:
    \begin{itemize}
    \item $u$ or $v$ has at least two neighbors in both $G_1$ and $G_2$. Note
      that this case can only happen if $k \ge 4$. By symmetry, assume that
      $u$ has at least two neighbors in each of $G_1$ and $G_2$. By assumption,
      each color but one is used exactly once in the neighborhood of $u$. As a
      result, there exists $c_1 \notin \alpha(N[u] \cap G_1)$ and $c_2 \notin
      \alpha(N[u] \cap G_2)$. The $(c_1, c_2)$-\Kcn containing the neighbor of
      $u$ colored $c_2$ is fully contained in $G_1$. After performing it, the
      color $c_2$ does not appear any more in the neighborhood of $u$ and one
      can conclude in one trivial \Kcg.
    \item $u$ has only one neighbor $u'$ in $G_1$ and $v$ only one neighbor $v'$
      in $G_2$ (or the other way around). If $u'$ and $v'$ are colored
      differently, the \Kcn $K_{u',\alpha(v')}(\alpha)$ is contains neither $u$,
      $v$ nor $v'$. Therefore, after performing the corresponding \Kcg, $u'$ and
      $v'$ are both colored $c'$. Consider a third color $c''$, $u$ has no
      neighbor colored $c''$ in $G_1$ and $v$ has no neighbor colored $c''$ in
      $G_2$, so $K_{u,c''}(\widetilde\alpha)$ does not contain $v$, and after
      performing it, $u$ and $v$ are colored differently.
    \end{itemize}
  \end{itemize}
\end{proof}

\subsection{$G$ 3-connected}
In this subsection, $G$ is a 3-connected $k$-regular graph of diameter. Given a
vertex $u \in V(G)$, we say that a pair of vertices $(t_1, t_2)$ is an eligible
pair of $u$ if $t_1$ and $t_2$ are two non-adjacent neighbors of $u$ and denote
the set of such pairs $P(u)=N(u)^2 \setminus E$.

\begin{lemma}[Adapted from Lemma 4.3 in~\cite{bonamy2019conjecture}]\label{lem:main}
  Assume that there exists two vertices $u$ and $x$ and an eligible pair
  $(t_1,t_2)$ of $u$ such that for each eligible pair $(w_1,w_2)$ of $x$, there
  exist a $k$-coloring that colors $t_1$ and $t_2$ alike and $w_1$ and $w_2$
  alike. That is :
  $$\forall (w_1,w_2) \in P(x), \mathcal{C}^k_{w_1,w_2}(G) \cap
  \mathcal{C}^k_{t_1,t_2}(G) \neq \emptyset.$$ Then $\mathcal{C}^k(G)$ forms a
  single Kempe class and $R^k(G)$ has diameter $O(n^2)$.
\end{lemma}
\begin{proof}
  Let $\alpha$ and $\beta$ be two colorings of $G$. Since $G$ is $k$-regular,
  there exists an eligible pair $(w_1,w_2)$ (respectively $(w_3, w_4)$) of $x$
  that is colored identically by $\alpha$ (respectively $\beta$). By assumption,
  there exists a $k$-coloring $\alpha' \in \mathcal{C}^k_{w_1,w_2}(G) \cap
  \mathcal{C}^k_{t_1,t_2}(G)$ and $\beta' \in \mathcal{C}^k_{w_3,w_4}(G) \cap
  \mathcal{C}^k_{t_1,t_2}(G)$. By applying~\cref{prop:blackbox} on $G_{w_1+w_2}$
  (respectively $G_{w_3,w_4}$), there exists a sequence of $O(n^2)$ Kempe
  changes between $\alpha$ and $\alpha'$ (respectively $\beta$ and
  $\beta'$). Finally by applying~\cref{prop:blackbox} in $G_{t_1+t_2}$, we get
  that $\alpha'$ is Kempe equivalent to $\beta'$ up to $O(n^2)$ Kempe changes.
  As a result, $\alpha$ and $\beta$ are equivalent up to $O(n^2)$ Kempe changes.
\end{proof}

\begin{lemma}[Adapted from Lemma 4.4 in~\cite{bonamy2019conjecture}]\label{lem:sepclique}
  If $G$ admits a vertex cut $S$ of size 3 such that one of the components of $G
  \setminus S$ is isomorphic to $K_k$. Then $R^k(G)$ has diameter at
  most $O(n^2)$.
\end{lemma}
\begin{proof}
  Let $C$ be the aforementioned component. Since $G$ is $k$-regular, each vertex
  of $C$ has exactly one neighbor in $S$ and all the others in $C$, namely, $S$
  weakly dominates $V(C)$. $|S| = 3$ and $d \ge 4$, hence at least one vertex
  $u$ of $S$ is adjacent to two or more vertices of $C$. Let $w_1$ be a neighbor
  of $u$ in $C$. Note that $S \setminus \{u\}$ is not a vertex cut of $G$,
  otherwise $G$ would not be 3-connected, therefore there exists a neighbor
  $w_2$ of $u$ that does not belong to $C$. By~\cref{prop:blackbox}, $R^k(G_{w_1+
    w_2})$ has diameter $O(n^2)$ and so does $R^k_{w_1,w_2}(G)$. Thus, if we
  prove that from any coloring $\alpha$ of $G$, one can reach a coloring where
  $w_1$ and $w_2$ are colored alike with a bounded number of moves, then we
  obtain $\diam(R^k(G)) = O(n^2)$.

  Let $\alpha$ be a $k$-coloring of $G$ such that $\alpha(w_1) \neq
  \alpha(w_2)$. Since $C$ is a clique on $d$ vertices, one of them is colored
  $\alpha(w_2)$, say $w_3$. Two cases may occur:
  \begin{itemize}
  \item If $w_3$ is adjacent to $u$, then it is not adjacent to any other vertex
    of $S$ and $\{w_1, w_3\}$ forms a \Kcn. After performing the corresponding
    \Kcg, $w_1$ and $w_2$ are colored alike.
  \item Otherwise, let $v$ be the neighbor of $w_3$ in $S$. From what precedes,
    $u$ has at least two neighbors in $C$ so one of them, say $w_4$, is colored
    differently from $v$ (it is possible that $w_4 = w_1$). Then $\{w_3, w_4\}$
    forms a \Kcn and after performing the corresponding move, either we had $w_4
    = w_1$ thus $w_1$ and $w_2$ are now colored alike, or $w_4$ and $w_1$ were
    distinct and we are now in the first case.
  \end{itemize}
  This proves that any coloring of $G$ is within at most 2 \Kcgs of a coloring
  in which $w_1$ and $w_2$ are colored alike, thus $\diam(R^k(G)) = O(n^2)$ by
  applying~\cref{prop:blackbox}.
\end{proof}

\begin{lemma}[Adapted from Lemma 4.5 in~\cite{bonamy2019conjecture}]\label{lem:config}
  Let $u, v$ be two vertices of $G$ and $(w_1,w_2)$ be an eligible
  pair in $P(v)$ such that neither $w_1$ nor $w_2$ is adjacent to $u$. Assume
  that there exists an eligible pair $(t_1,t_2)$ in $P(u)$ such that there is no
  $k$-coloring of $G$ that colors $w_1$ and $w_2$ alike and $t_1$ and $t_2$
  alike. Then $G$ contains a subgraph weakly dominated by both $\{t_1,t_2\}$ and
  $\{w_1, w_2\}$, that is isomorphic $K_{k-1}$.
\end{lemma}
\begin{proof}
  The proof of Lemma 4.5 presented in~\cite{bonamy2019conjecture}, results in
  the same outcome by assuming the stronger hypothesis that $\mathcal{C}^k(G)$
  does not form a single Kempe class, but it only uses the fact that there is no
  $k$-coloring of $G$ that colors $w_1$ and $w_2$ alike and $t_1$ and $t_2$
  alike, therefore we can use it directly.
\end{proof}

\begin{proposition}[Adapted from Lemma 4.6 in~\cite{bonamy2019conjecture}]\label{prop:diam3}
  Assume that there are two non-adjacent vertices $u, v$ of $G$ and an eligible
  pair $(w_1, w_2)$ in $P(v)$ such that neither $w_1$ nor $w_2$ is adjacent to
  $u$ -- note that this is in particular the case in graphs of diameter at least
  3. Then $C^k(G)$ forms a single Kempe class and $R^k(G)$ has diameter
  $O(n^2)$.
\end{proposition}
\begin{proof}
  If for all eligible pair $(t_1,t_2)$ in $P(u)$, $\mathcal{C}^k_{t_1, t_2}(G)
  \cap \mathcal{C}^k_{w_1,w_2}(G) \neq \emptyset$, then the results holds
  by~\cref{lem:main}.
  Otherwise, there exist an eligible
  pair $(t_1,t_2)$ in $P(u)$ such that $\mathcal{C}^k_{w_1,w_2}(G) \cap
  \mathcal{C}^k_{t_1,t_2}(G) = \emptyset$. By~\cref{lem:config}, $G$ contains a
  subgraph $C$ isomorphic to $K_{k-1}$ and weakly dominated by both
  $\{t_1,t_2\}$ and $\{w_1,w_2\}$. Each vertex of $C$ is adjacent to the $k-2$
  others and to one vertex of $\{t_1,t_2\}$ and one of $\{w_1,w_2\}$. The clique
  $C$ does not contain $u$ nor $v$ since they are adjacent to both $\{t_1,t_2\}$
  and $\{w_1,w_2\}$.

  At least three vertices of $\{t_1, t_2, w_1, w_2\}$ are adjacent to at least
  one vertex of $C$, otherwise $G$ is not 3-connected. If exactly one vertex,
  say $t_1$, is not adjacent to a vertex in $C$, then $\{u, w_1, w_2\}$ is a cut
  set between the clique $C \cup \{t_2\}$ and the rest of the graph, so one can
  apply~\cref{lem:sepclique}. We will show that this is the only possible case,
  thereby completing the proof.

  Assume towards contradiction that all the vertices of $\{t_1, t_2, w_1, w_2\}$
  are adjacent to at least one vertex of $C$. Assume, without loss of generality
  that $w_1$ has at least as many neighbors in $C$ as $w_2$. Without loss of
  generality, assume that $t_1$ and $w_1$ have a common neighbor $x$ in
  $C$. Then $(x,v)$ is a eligible pair of $w_1$, such that neither $w_1$, nor
  $x$, nor $v$ is adjacent to $u$. If for all eligible pair $(t_3,t_4)$ in
  $P(u)$, $\mathcal{C}^k_{t_3, t_4}(G) \cap \mathcal{C}^k_{x,v}(G) \neq
  \emptyset$, we can once again conclude by~\cref{lem:main}.

  Otherwise, there exists an eligible pair $(t_3,t_4)$ of $u$ such that
  $\mathcal{C}^k_{x,v}(G) \cap \mathcal{C}^k_{t_3,t_4}(G) =
  \emptyset$. By~\cref{lem:config}, there exists a $(k-1)$-clique $C'$ in $G$
  weakly dominated by both $\{t_3,t_4\}$ and $\{x,v\}$. With the same reasoning,
  each of the vertices $\{t_3, t_4, x, v\}$ is adjacent to at least one vertex
  of $C'$.

  Assume towards contradiction that neither $t_1$ nor $t_2$ belong to $C'$. Then
  $N(x) = C \setminus \{x\} \{t_1,w_1\}$ and at least one neighbor of $x$
  belongs to $C$. By assumption, $t_1 \notin C'$ and $w_1$ neither since it is
  adjacent to both $x$ and $v$. So there is a vertex $y \neq x \in C \cup C'$
  and $(k-2)$ other neighbors of $y$ belong to $C'$. As none of $\{t_1, t_2,
  w_1, x\}$ belong to $C'$, we have $C' = C \setminus\{x\} \cup \{w_2\}$ so
  $w_2$ is adjacent to all the vertices of $C \setminus\{x\}$. By assumption,
  $w_1$ has at least as many neighbors in $C$ as $w_2$ so $C$ must be a clique
  on 2 vertices, and $k = 3$ which is a contradiction and proves that, exactly
  one of $t_1$ and $t_2$ belong to $C'$ ($t_1$ and $t_2$ are non-adjacent by
  definition).

  For $i \in \{1,2\}$, if $t_i$ belongs to $C'$, then it has $k-2$ neighbors in
  $C'$ and two other neighbors. By definition, $t_i$ is adjacent $u$ and to one
  of $\{x,v\}$. But neither of $t_3$ and $t_4$ belongs to $C' \cup \{u, x, v\}$
  which contradicts the definition of $C'$ and proves that $\{t_1, t_2, w_1,
  w_2\}$ cannot all have at least one neighbor in $C$. This concludes the proof.
\end{proof}

We denote $N^2(u)$ the \emph{second neighborhood} of a vertex $u$, that is the
set of vertices at distance exactly two from $u$.

\begin{proposition} Let $G$ be a 3-connected $k$-regular graph of diameter at
  most 2, with $k \ge 4$. Then $\mathcal{C}^k(G)$ forms a single Kempe class and
  $\diam(R^k(G)) )= O(n^2)$.
\end{proposition}
\begin{proof}
  If $G$ is of diameter less than 2, the result is obvious, thus we can assume
  that $\diam(G) = 2$. If there exists two non-adjacent vertices $u,v$ in $G$ and
  an eligible pair $(w_1,w_2)$ in $P(v)$, such that neither $w_1$ nor $w_2$ is
  adjacent to $u$, then the results holds by~\cref{prop:diam3}. Otherwise, the
  second neighborhood of any vertex can contain no path on three vertices, thus
  is a collection of disjoint cliques.

  Assume that the second neighborhood of a vertex $v$ consists of at least two
  disjoint cliques $C_1$ and $C_2$. Let $x$ and $y$ be two vertices of $C_1$ and
  $C_2$ respectively. If $x$ is adjacent to a neighbor $z$ of $v$ that is not
  adjacent to $y$, then $\{x, y, z\}$ induce a path in the second neighborhood
  of $y$, and one can conclude by~\cref{prop:diam3}. As a result, we can assume
  that $N(x) \cap N(v) = N(y) \cap N(v)$. By repeating this argument for all of
  pair of vertices in $C_1 \times C_2$, we obtain that all the vertices in $C_1
  \cup C_2$ have the same set of neighbors in $N(v)$. Given a coloring $\alpha$
  of $G$, if $x$ and $y$ have distinct colors, say 1 and 2, the $\{1, 2\}$ Kempe
  chain containing $x$ cannot contain any vertices of $C_2$, thus, after
  performing it, $x$ and $y$ are colored identically. Since they have at least
  one common neighbor, we can conclude by~\cref{prop:blackbox}.

  Therefore, we can assume without loss of generality that the second
  neighborhood of each vertex consist in a single clique. Let $v \in V$ and let
  $\alpha$ and $\beta$ be two colorings of $G$. Denote by $C$ the second
  neighborhood of $v$. Up to a Kempe change, one can assume that $\alpha(v) =
  \beta(v) = 1$.

  If the color 1 is not used by $\alpha$ in $C$, then let $x$ be the vertex of
  $C$ colored 1 in $\beta$ (or any vertex in $C$ if the color is also not used
  by $\beta$ in $C$). The only vertex of $C$ in the Kempe chain $K_{x,1}(\alpha,
  G)$ is $x$, and as no vertices of $N(v)$ can be colored 1, $K_{x,1}(\alpha,
  G)$ is a trivial Kempe change. We can apply it to color identically $x$ and
  $v$ with color 1. If no vertex of $C$ is colored 1 by $\beta$, we can recolor
  $x$ with 1 using the same argument, without changing the color of
  $v$. By~\cref{prop:blackbox}, the two colorings obtained are equivalent up to
  at most $O(n^2)$ Kempe changes.

  Therefore, one can assume without loss of generality that there exists two
  vertices $u,w$ in $C$ such that $\alpha(u) = \beta(w) = 1$. Once again if $u =
  w$, we can conclude by~\cref{prop:blackbox}, so one can assume that $u \neq
  w$. We will require the following definition: a vertex $x$ is said to be
  \emph{locked} in $\gamma$ if all the colors are present in $N(x)$, making a
  trivial Kempe change on $x$ impossible.

  \begin{claim}\label{cl:locked}
    If any of the vertices in $N[u] \setminus\{w\}$ is not locked or if
    the color $\alpha(w)$ does not appear twice in the neighborhood of $u$,
    then $R^k(G)$ is connected and of diameter $O(n^2)$.
  \end{claim}
  If $u$ is not locked, then a trivial Kempe change on $u$ removes the color 1
  from $C$ and we can conclude with the aforementioned argument. So one can
  assume that $u$ is locked. If $v \notin K_{u, \alpha(w)}(\alpha,G)$, then
  performing the corresponding Kempe change yields a coloring in which both $w$
  and $v$ are colored 1 and since this is also the case in $\beta$, one can
  conclude with~\cref{prop:blackbox}. Assume that $v \in K_{u,
    \alpha(w)}(\alpha,G)$, as all the vertices are adjacent to $u$ or $v$, $w$
  cannot be adjacent to any vertex colored 1 other than $u$, so there must be a
  vertex $y$ adjacent to both $v$ and $u$ colored $\alpha(w)$. To sum up, all
  the colors except color 1 are present in the neighborhood of $u$, and the
  color $\alpha(w)$ is used by both $w$ and $y$.

  If $y$ is not locked, one can perform a trivial Kempe change on $y$, after
  which $v \notin {K_{u,\alpha(w)}(\widetilde\alpha, G)}$ and we are back it the
  previous situation. If any vertex of $N(u) \setminus \{w, y\}$ is not locked,
  after a trivial Kempe change, $u$ is not locked anymore.

  This proves the claim. From now on, we assume that $v$ belongs to $K_{u,
    \alpha(w)}(\alpha,G)$, that the vertices of $N[u] \setminus \{w\}$ are
  locked in $\alpha$, and that the color $\alpha(w)$.  We make the symmetric
  assumptions for $\beta$. To conclude this proof, we discuss the two following
  cases :
  \begin{itemize}
  \item If $|C| \ge 3$, let $z \in C \setminus \{u,w\}$. Each vertex in $G$ is
    adjacent to $u$ or $v$, so $u$ is the only neighbor of $w$ colored 1 in
    $\alpha$. Likewise, $w$ is the only neighbor of $u$ colored 1 in $\beta$. By
    assumption, $z$ is the only neighbor of $u$ colored $\alpha(z)$ in $\alpha$,
    so $K_{z,1}(\alpha, G) = \{z, u\}$. The same argument applied in $\beta$
    gives us $K_{z,1}(\beta, G) = \{z, w\}$. By performing each of these Kempe
    changes, we obtain two colorings that both color $z$ and $v$ with color 1
    and we can conclude with~\cref{prop:blackbox}.
  \item Otherwise, if $|C| = 2$, $G$ contains $v$, its $k$ neighbors and $u$ and
    $w$. The vertices $u$ and $w$ are adjacent to one another and are each
    adjacent to $k-1$ neighbors of $v$. So $S = N(u) \cap N(w)$ contains at
    least $k-2$ vertices. By assumption, there exists a vertex $z$ in $N(u) \cap
    N(v)$ that is colored $\alpha(w)$ by $\alpha$ and thus is not adjacent to
    $w$. Likewise, there exists a vertex $z'$ in $N(w) \cap N(v)$ that is
    colored $\beta(u)$ by $\beta$ and thus is not adjacent to $u$. As a result,
    $S$ contains exactly $k-2$ vertices. By assumption, all the vertices of $S$
    are locked in $\alpha$ and since both $v$ and $u$ are colored 1 in $\alpha$,
    the vertices of $S$ each have exactly one neighbor of each other color. But
    the vertices of $S$ are adjacent to $w$, so they cannot be adjacent to $z$
    which is colored alike. So $z$ has at most three neighbors : $v$, $w$ and
    $z'$, contradicting $k \ge 4$. As a result, the assumptions
    of~\cref{cl:locked} are always met when $|C| = 2$.
  \end{itemize}
\end{proof}

\end{document}